\documentclass{amsart}
\usepackage{amsfonts,amscd,amsthm,amsgen,amsmath,amssymb}
\usepackage[all]{xy}
\usepackage{epsfig,color}
\usepackage[vcentermath]{youngtab}
\newtheorem{theorem}{Theorem}[section]

\newtheorem{corollary}[theorem]{Corollary}
\newtheorem{proposition}[theorem]{Proposition}

\theoremstyle{definition}
\newtheorem{definition}[theorem]{Definition}

\theoremstyle{remark}
\newtheorem{remark}[theorem]{Remark}

\numberwithin{equation}{section}

\begin{document}

\title[Constraints from Bauer-Furuta invariants]{Constraints on families of smooth 4-manifolds from Bauer-Furuta invariants}
\author{David Baraglia}

\address{School of Mathematical Sciences, The University of Adelaide, Adelaide SA 5005, Australia}

\email{david.baraglia@adelaide.edu.au}


\date{\today}

\begin{abstract}
We obtain constraints on the topology of families of smooth $4$-manifolds arising from a finite dimensional approximation of the families Seiberg-Witten monopole map. Amongst other results these constraints include a families generalisation of Donaldson's diagonalisation theorem and Furuta's $10/8$ theorem. As an application we construct examples of continuous $\mathbb{Z}_p$-actions for any odd prime $p$, which can not be realised smoothly. As a second application we show that the inclusion of the group of diffeomorphisms into the group of homeomorphisms is not a weak homotopy equivalence for any compact, smooth, simply-connected indefinite $4$-manifold with signature of absolute value greater than $8$.

\end{abstract}

\maketitle


\section{Introduction}

In a previous paper \cite{bar} we showed how the moduli space of the Seiberg-Witten equations for a smooth family of $4$-manifolds imposes constraints on the topology of the family. In this paper we instead consider a finite dimensional approximation of the Seiberg-Witten monopole map and again obtain constraints on the topology of the family. There are two main advantages compared to the previous approach. Firstly, the constraints that we obtain from the monopole map are generally stronger than those obtained from the families moduli space. Secondly, the monopole map approach allows us to bypass certain transversality issues which arise in the construction of the families moduli space. On the other hand, there are results in \cite{bar} that we have not been able to recover using the Seiberg-Witten monopole map, so it would appear that the two approaches complement one another.\\

The setting that we are interested in is as follows: let $X$ be a compact, oriented, smooth $4$-manifold and let $\mathfrak{s}$ be a spin$^c$-structure on $X$. Consider a family of $4$-manifolds over a compact smooth base manifold $B$ with fibres diffeomorphic to $X$. In other words consider a smooth locally trivial fibre bundle $\pi : E \to B$ over $B$ whose fibres are diffeomorphic to $X$. Suppose that $E$ is equipped with a fibrewise orientation and fibrewise spin$^c$-structure $\mathfrak{s}_{E/B}$ (that is, a spin$^c$-structure on the vertical tangent bundle) which restricts to $\mathfrak{s}$ on each fibre. We will say that $(E , \mathfrak{s}_{E/B})$ is a {\em spin$^c$-family over $B$ with fibre $(X , \mathfrak{s})$}. Similarly if the vertical tangent bundle is equipped with a spin structure then we may speak of a {\em spin family over $B$}. The topological conditions for the existence of a fibrewise spin or spin$^c$-structure extending a given spin or spin$^c$-structure on $X$ is studied at length in \cite[Section 2.1]{bar}.\\

Given a spin$^c$-family $(E , \mathfrak{s}_{E/B})$ we can associate two topological invariants:
\[
H^+(X) \in KO^0(B), \quad D \in K^0(B),
\]
where $H^+(X)$ is the vector bundle whose fibre over $b \in B$ is the space of harmonic self-dual $2$-forms on the corresponding fibre of $E$ (with respect to some choice of smoothly varying fibrewise metric on $E$) and $D$ is the families index of the spin$^c$ Dirac operator of the family $(E , \mathfrak{s}_{E/B})$. By studying the Seiberg-Witten equations of the family $(E , \mathfrak{s}_{E/B})$, or more precisely, by considering a finite dimensional approximation of the Seiberg-Witten equations, we obtain non-trivial constraints on the topology of the classes $H^+(X)$ and $D$. In turn this implies topological constraints for the existence of a spin$^c$ family $(E , \mathfrak{s}_{E/B})$ to realise the pair $(H^+(X) , D)$.\\

In Sections \ref{sec:setup}-\ref{sec:spin}, we consider only $4$-manifolds with $b_1(X)=0$. However in Section \ref{sec:b1>0} we see that the results of those sections also hold for $b_1(X)>0$ without any additional assumptions. Below we summarise the main results of the paper.\\

Let $b_+(X)$ denote the rank of $H^+(X)$ and $d$ the rank of $D$. Then $d = (c_1(\mathfrak{s})^2 - \sigma(X))/8$, where $\sigma(X)$ is the signature of $X$. Our first result can be thought of as the families Seiberg-Witten generalisation of Donaldson's diagonalisation theorem:

\begin{theorem}
Let $(E , \mathfrak{s}_{E/B})$ be a spin$^c$-family over $B$ with fibre $(X , \mathfrak{s})$. 
\begin{itemize}
\item{If the Euler class $e(H^+(X))$ of $H^+(X)$ is non-zero, then $c_1(\mathfrak{s})^2 \le \sigma(X)$. Moreover $e(H^+(X))s_j(D)=0$ whenever $j > -d$, where $s_j(D)$ is the $j$-th Segre class of $D$ (see Section \ref{sec:cohomology}).}
\item{If the $K$-theoretic Euler class of $H^+(X)_\mathbb{C} = H^+(X) \otimes_{\mathbb{R}} \mathbb{C}$ is non-zero then $c_1(\mathfrak{s})^2 \le \sigma(X)$.}
\end{itemize}
\end{theorem}

In the case that the spin$^c$-structure of the family $(E , \mathfrak{s}_{E/B})$ comes from a spin structure, we say that $(E , \mathfrak{s}_{E/B})$ is a {\em spin family}. Using the $Pin(2)$-symmetry of the Seiberg-Witten equations for spin structures, we obtain the following results:

\begin{theorem}
Let $(E , \mathfrak{s}_{E/B})$ be a spin family over $B$ with fibre $(X , \mathfrak{s})$. If the $i$-th Steifel-Whitney class $w_i( H^+(X) )$ of $H^+(X)$ is non-zero for some $i \in \{b_+(X), b_+(X)-1,b_+(X)-2 \}$, then $\sigma(X) \ge 0$.
\end{theorem}

The next result is the families Seiberg-Witten generalisation of Furuta's $10/8$ theorem:

\begin{theorem}
Let $(E , \mathfrak{s}_{E/B})$ be a spin family over $B$ with fibre $(X , \mathfrak{s})$. Then there exist complex vector bundles $V,V'$ on $B$ such that $D = [V] - [V']$ and
\[
\wedge^* H^+(X)_{\mathbb{C}} \otimes \wedge^* \psi^2(V') =  \eta  (\wedge^* \psi^2(V))
\]
for some $\eta \in K^0(B)$, where $\psi^2$ denotes the $2$nd Adams operation. Moreover, if $e^K(H^+(X)_{\mathbb{C}})=0$ then there exists $\eta' \in K^0(B)$ such that in $K^0(B)/torsion$ we have:
\[
\wedge^* H^+(X)_{\mathbb{C}} \otimes \wedge^* \psi^2(V') =  2 \eta'  (\wedge^* \psi^2(V)).
\]
\end{theorem}

In Section \ref{sec:equivariant}, instead of families we consider the $G$-equivariant Seiberg-Witten monopole map for a finite group $G$ acting smoothly on a $4$-manifold $X$ equipped with a lift of the action to the spin bundles of $(X,\mathfrak{s})$. We obtain $G$-equivariant analogues of the above theorems. In Sections \ref{sec:z2}-\ref{sec:zp}, we specialise to the case the $G$ is a finite cyclic group of prime order. Already in this case our main results imply some interesting non-trivial constraints for actions of finite cyclic groups on $4$-manifolds.\\

Consider first the case of smooth $\mathbb{Z}_2$-actions. Let $f : X \to X$ be the generator and suppose $f$ preserves the isomorphism class of a spin$^c$-structure $\mathfrak{s}$ on $X$. Then we can choose a lift $\tilde{f}$ of $f$ to the associated spinor bundles satisfying $\tilde{f}^2 = 1$ and this lift is unique up to an overall sign change $\tilde{f} \mapsto -\tilde{f}$. Let $d_{\pm}$ denote the virtual dimensions of the $\pm 1$ virtual eigenspaces of $\tilde{f}$ on $D$. Choose an $f$-invariant metric so that $\mathbb{Z}_2 = \langle f \rangle$ acts on $H^+(X)$. We let $H^+(X)^{\mathbb{Z}_2}$ denote the invariant subspace.

\begin{theorem}
Suppose that $H^+(X)^{\mathbb{Z}_2} = 0$. Then for any $f$-invariant spin$^c$-structure $\mathfrak{s}$, we have $d_+,d_- \le 0$.
\end{theorem}

Now suppose that $X$ is spin and that $f$ preserves a spin structure $\mathfrak{s}$. Recall that $f$ is said to be of even type if it can be lifted to an involution on the associated principal $Spin(4)$-bundle. If the fixed point set of $f$ is non-empty then $f$ is even if and only if its fixed point set is discrete \cite[Proposition 8.46]{ab}.

\begin{theorem}
Suppose that $f$ preserves a spin structure $\mathfrak{s}$ and $f$ is of even type with respect to $\mathfrak{s}$. If $\sigma(X) < 0$, then $dim( H^+(X)^{\mathbb{Z}_2}) \ge 3$.
\end{theorem}

By way of comparison we note that by a theorem of Bryan \cite[Theorem 1.5]{bry}, if $f$ preserves a spin structure (where $f$ is not necessarily even) and $\sigma(X) < 0$, then $dim( H^+(X)^{\mathbb{Z}_2}) \ge 1$.\\

Now we consider $\mathbb{Z}_p$-actions for an odd prime $p$. Let $f : X \to X$ be the generator and suppose $f$ preserves the isomorphism class of a spin$^c$-structure $\mathfrak{s}$. Then we can choose a lift $\tilde{f}$ of $f$ to the associated spinor bundles satisfying $\tilde{f}^p = 1$. Such a lift is uniquely determined up to multiplication by a $p$-th root of unity. For $0 \le j \le p-1$ we let $d_j$ denote the dimension of the $\omega^j$ virtual eigenspace of $\tilde{f}$ on $D$ where $\omega = exp(2\pi i /p)$.

\begin{theorem}
Suppose that $H^+(X)^{\mathbb{Z}_p} = 0$. Then for any $f$-invariant spin$^c$-structure $\mathfrak{s}$, we have $d_j \le 0$ for each $j$.
\end{theorem}

We consider an application of these results to non-smoothability of certain continuous $\mathbb{Z}_p$-actions. Denote by $E_8$ the unique compact simply connected topological $4$-manifold with intersection form $E_8$ and let $-E_8$ denote the same manifold with the opposite orientation.

\begin{theorem}
Let $p$ be an odd prime and let $b,g$ be integers with $g(p-1) \ge 3bp$ and $b \ge 1$. Let $X$ be the topological $4$-manifold $X = \# g(p-1) (S^2 \times S^2) \# 2bp(-E_8)$. Then $H^2(X ; \mathbb{Z})$ admits an isometry $f : H^2(X ; \mathbb{Z}) \to H^2(X ; \mathbb{Z})$ of order $p$ with the following properties:
\begin{itemize}
\item[(i)]{$f$ can be realised by the induced action of a continuous, locally linear $\mathbb{Z}_p$-action on $X$.}
\item[(ii)]{If $g(p-1) > 3bp$ then $f$ can be realised by the induced action of a diffeomorphism $X \to X$, where the smooth structure is obtained by viewing $X$ as $\# (g(p-1)-3bp)(S^2 \times S^2) \# pb(K3)$.}
\item[(iii)]{$f$ can not be induced by a smooth $\mathbb{Z}_p$-action for any smooth structure on $X$.}
\end{itemize}
\end{theorem}

We finish off the paper with an application of our obstruction theorems to the existence of non-smoothable families. For a smooth $4$-manifold $X$ we let $Homeo(X)$ denote the group of homeomorphisms of $X$ with the $\mathcal{C}^0$-topology and $Diff(X)$ the group of diffeomorphisms of $X$ with the $\mathcal{C}^\infty$-topology. This application was suggested to the author by Hokuto Konno and generalises a result of Kato-Konno-Nakamura \cite[Corollary 1.5]{kkn}:
\begin{theorem}\label{thm:nonsmooth}
Let $X$ be a compact, smooth, simply-connected $4$-manifold with $| \sigma(X) | > 8$ and indefinite intersection form. Then:
\begin{itemize}
\item{If $X$ is non-spin, there exists a topological fibre bundle $E \to B$ with fibres homeomorphic to $X$ and $B$ is a torus of dimension $\min \{b_+(X),b_-(X)\}$ such that $E$ is non-smoothable, i.e. $E$ does not admit a reduction of structure to $Diff(X)$.}
\item{If $X$ is spin, there exists a topological fibre bundle $E \to B$ with fibres homeomorphic to $X$ and $B$ is a torus of dimension $\min \{b_+(X),b_-(X)\}-2$ such that $E$ is non-smoothable.}
\end{itemize}
\end{theorem}

Using an obstruction theory argument we obtain the following corollary:

\begin{corollary}
Let $X$ be a compact, smooth, simply-connected $4$-manifold with $| \sigma(X) | > 8$ and indefinite intersection form. Then the inclusion $Diff(X) \to Homeo(X)$ is not a weak homotopy equivalence. More precisely:
\begin{itemize}
\item{If $X$ is non-spin then $\pi_j( Diff(X)) \to \pi_j( Homeo(X))$ is not an isomorphism for some $j \le \min\{b_+(X),b_-(X)\}-1$.}
\item{If $X$ is spin then $\pi_j( Diff(X)) \to \pi_j( Homeo(X))$ is not an isomorphism for some $j \le \min\{b_+(X),b_-(X)\}-3$.}
\end{itemize}
\end{corollary}

\begin{remark}
Consider the case where $X$ is a compact, smooth, simply-connected $4$-manifold with $b_+(X)=1$ and $b_-(X) \ge 10$. Then $\pi_0(Diff(X)) \to \pi_0(Homeo(X))$ is not an isomorphism. In fact, Theorem \ref{thm:nonsmooth} gives a non-smoothable family over the circle. Such a family is the mapping cylinder of a homeomorphism $f : X \to X$ and non-smoothability implies that $f$ is not isotopic to a diffeomorphism. We deduce that $\pi_0(Diff(X)) \to \pi_0(Homeo(X))$ is not surjective for such $4$-manifolds.
\end{remark}

A brief outline of the contents of this paper is as follows: in Section \ref{sec:setup} we describe the setup to be used in this paper and recall that a finite dimensional approximation of the families Seiberg-Witten monopole map can be constructed. In Section \ref{sec:cohomology} we study the action of the monopole map on $S^1$-equivariant cohomology to obtain non-trivial topological constraints on the family. In Section \ref{sec:ktheory} we consider the action of the monopole map on $S^1$-equivariant $K$-theory. In Section \ref{sec:spin}, we consider spin families, in which case the monopole map has $Pin(2)$ symmetry. We study the action of the monopole map on $Pin(2)$-equivariant cohomology and $K$-theory. In Section \ref{sec:equivariant} we adapt our setup to the equivariant setting with respect to a finite group. In Section \ref{sec:b1>0} we show how the results of the previous sections can be carried over to the case $b_1(X)>0$. We consider $\mathbb{Z}_2$-actions in Section \ref{sec:z2} and $\mathbb{Z}_p$-actions in Section \ref{sec:zp}. Finally, we consider an application to the existence of non-smoothable families in Section \ref{sec:nonsmooth}.\\

\noindent{\bf Acknowledgments} The author wishes to thank Hokuto Konno for helpful comments and in particular suggesting the application to the existence of non-smoothable families as in Theorem \ref{thm:nonsmooth}. The author also would like to thank the referee who helped to improve the results of this paper. The author was financially supported by the Australian Research Council Discovery Early Career Researcher Award DE160100024 and Australian Research Council Discovery Project DP170101054.

\section{Setup}\label{sec:setup}

Let $X$ be a compact, oriented, smooth $4$-manifold. For the time being we will assume that $b_1(X)=0$. The case $b_1(X)>0$ will be dealt with in Section \ref{sec:b1>0}. Let $\mathfrak{s}$ be a spin$^c$-structure on $X$ with characteristic $c = c_1(\mathfrak{s}) \in H^2(X ; \mathbb{Z})$. Let $d = (c^2 - \sigma(X))/8$ be the index of the associated spin$^c$ Dirac operator.\\

Let $B$ be a compact smooth manifold. As in the introduction, we consider a spin$^c$-family $(E , \mathfrak{s}_{E/B})$ over $B$ with fibres diffeomorphic to $(X , \mathfrak{s})$. This consists of a locally trivial fibre bundle $\pi : E \to B$ with fibres diffeomorphic to $X$ and additionally $E$ is equipped with a fibrewise orientation and fibrewise spin$^c$-structure $\mathfrak{s}_{E/B}$ which restricts to $\mathfrak{s}$ on each fibre.\\

Let $S^1$ act on $\mathbb{C}$ by scalar multiplication and trivially on $\mathbb{R}$. Taking a finite dimensional approximation of the Seiberg-Witten monopole map for $(X,\mathfrak{s})$ gives an $S^1$-equivariant map
\[
f : (\mathbb{C}^{a} \oplus \mathbb{R}^b)^+ \to (\mathbb{C}^{a'} \oplus \mathbb{R}^{b'})^+
\]
for some $a,b,a',b' \ge 0$, where $a-a' = d$, $b'-b = b_+(X)$. Here $T^+$ denotes the one-point compactification of $T$. The stable equivariant homotopy class of $f$ defines the Bauer-Furuta invariant of $(X , \mathfrak{s})$ \cite{bafu} and it can be used to recover the Seiberg-Witten invariant of $(X , \mathfrak{s})$ when $b_+(X) > 1$. If $\mathfrak{s}$ comes from a spin structure, then the finite dimensional approximation of the monopole map can be taken equivariant with respect to the larger group $Pin(2) = S^1 \cup jS^1$:
\[
f : (\mathbb{H}^{a} \oplus \mathbb{R}^b_-)^+ \to (\mathbb{H}^{a'} \oplus \mathbb{R}^{b'}_{-})^+
\]
where $a-a' = d/2 = -\sigma(X)/16$, $b'-b = b_+(X)$ and where $\mathbb{R}_-$ is the representation of $Pin(2)$ such that $S^1$ acts trivially and $j$ acts as multiplication by $-1$.\\

An important property of the finite dimensional approximation is that $f$ can be chosen so that its restriction $f|_{(\mathbb{R}^b)^+} : (\mathbb{R}^b)^+ \to (\mathbb{R}^{b'})^+$ is the map induced by an inclusion of vector spaces $\mathbb{R}^b \subseteq \mathbb{R}^{b'}$. The existence of such a map already implies non-trivial conditions on $X$. For instance, if $b_+(X) = 0$ then one can show that $d \le 0$, so that $c_1(\mathfrak{s})^2 \le \sigma(X)$ for any spin$^c$-structure on $X$. This inequality is known to imply that the intersection form on $X$ is diagonal, which is Donaldson's diagonalisation theorem.\\

In the same way, we will show that the existence of a finite dimensional approximation of the monopole map for a spin$^c$-family $(E , \mathfrak{s}_{E/B})$ implies non-trivial conditions on the topology of the family. In the families setting, the finite dimensional approximation $f$ is replaced by a family of such maps parametrised by $B$. To formulate this properly we introduce some notation. Let $V,V'$ be complex vector bundles over $B$ of ranks $a,a'$ and let $U,U'$ be a real vector bundles over $B$ of ranks $b,b'$. We let $S^1$ act on $V,V'$ by scalar multiplication in the fibres and act trivially on $U,U'$. The finite dimensional approximation of the families Seiberg-Witten monopole map is an $S^1$-equivariant map of sphere bundles
\[
f : S_{V , U} \to S_{V' , U'},
\]
where $S_{V,U}, S_{V',U'}$ denote the fibrewise one-point compactifications of $V\oplus U$ and $V' \oplus U'$ respectively (a detailed construction of the map $f$ is given in \cite{bako}. See also \cite{szy1}). Moreover, the following relations hold in $K^0(B)$ and $KO^0(B)$ respectively:
\[
V - V' = D, \quad \quad U' - U = H^+(X),
\]
where $D \in K^0(B)$ is the families index of the fibrewise spin$^c$ Dirac operator of the family $(E , \mathfrak{s}_{E/B})$ and $H^+(X)$ is the vector bundle on $B$ whose fibre over $b \in B$ is the space of harmonic self-dual $2$-forms on the fibre of $E$ over $b$ (after fixing a choice of smoothly varying fibrewise metric on $E$). By stabilising, we can assume that $V',U'$ are trivial vector bundles, or alternatively, that $V,U$ are trivial. Furthermore, since $H^+(X)$ is a genuine vector bundle we can even assume that $U' = U \oplus H^+(X)$.\\

If the families spin$^c$-structure $\mathfrak{s}_{E/B}$ can be lifted to a families spin structure, then we can take $V,V'$ to be quaternionic vector bundles and we may take $f$ to be $Pin(2)$-equivariant, where $Pin(2)$ acts on $V,V'$ according to the quaternionic structures, $j$ acts on $U,U'$ as multiplication by $-1$ and $S^1$ acts trivially on $U,U'$.\\

There are two further properties of the map $f$ that we need to use in order to extract useful results, both of which are shown in \cite{bako}. First, we may assume that $f|_{S_U} : S_U \to S_{U'}$ is the map induced by the inclusion $U \to U' = U \oplus H^+(X)$. Second, we may assume that $f$ sends the point at infinity in any fibre of $S_{V,U}$ to the point at infinity of the corresponding fibre of $S_{V',U'}$. Stated differently, we let $B_{V,U} \subseteq S_{V,U}$ denote union of the points at infinity in each fibre and similarly define $B_{V',U'} \subseteq S_{V',U'}$. Then $f$ defines an $S^1$-equivariant map of pairs
\[
f : (S_{V,U} , B_{V,U} ) \to (S_{V',U'} , B_{V',U'}).
\]
As we will see in the following sections, the existence of such a map $f$ imposes non-trivial constraints on the classes $D \in K^0(B)$, $H^+(X) \in KO^0(B)$, which in turn imposes constraints on the topology of the family $E \to B$ itself.

\section{Cohomological constraints}\label{sec:cohomology}
Suppose we are in the setup of Section \ref{sec:setup}, so we have a spin$^c$-family $(E , \mathfrak{s}_{E/B})$ from which we obtain a finite dimensional approximation of the families monopole map:
\[
f : (S_{V,U} , B_{V,U} ) \to (S_{V',U'} , B_{V',U'}).
\]
where $f$ is $S^1$-equivariant. Denote the complex ranks of $V,V'$ as $a,a'$ and the real ranks of $U,U'$ as $b,b'$. Then $a = a' + d$, $b' = b + b_+(X)$. Given an $S^1$-equivariant cohomology theory $E^*_{S^1}$, one can consider the induced map 
\[
f^* : E^*_{S^1}( S_{V , U} , B_{V , U} ) \to E^*_{S^1}( S_{V' , U'} , B_{V' , U'} ).
\]
In this section, we take $E^*_{S^1} = H^*_{S^1}$ to be given by $S^1$-equivariant cohomology. In the following section we will consider equivariant $K$-theory.\\

Let $i : S_U \to S_{V , U}$ and $j : S_{U'} \to S_{V' , U'}$ be the inclusion maps. We have a commutative diagram of pairs
\begin{equation*}\xymatrix{
(S_{V , U} , B_{V,U}) \ar[r]^-f & (S_{V' , U'} , B_{V',U'}) \\
(S_U , B_U) \ar[u]^-i \ar[r]^-{f|_{S_U}} & (S_{U'} , B_{U'}) \ar[u]^-j
}
\end{equation*}
Moreover we have seen that $U'$ can be taken to be $U \oplus H^+(X)$ and $f|_{S_U}$ to be the inclusion $\iota : U \to U'$. We will also stabilise $f$ in such a way that $V,U$ are trivial bundles: $V = \mathbb{C}^a$, $U = \mathbb{R}^b$.\\

Recall that $H^*_{S^1}(pt ; \mathbb{Z}) \cong \mathbb{Z}[x]$, where $x \in H^2_{S^1}(pt ; \mathbb{Z})$ is the first Chern class of the universal principal circle bundle $ES^1 \to BS^1$. Then since $S^1$ acts trivially on $B$, we have $H^*_{S^1}(B ; \mathbb{Z}) \cong H^*(B ; \mathbb{Z})[x]$. For any real oriented vector bundle $W \to B$ of rank $r$, equipped with the trivial circle action, the equivariant cohomology $H^*_{S^1}( S_W , B_W ; \mathbb{Z})$ is a free $H^*(B ; \mathbb{Z})[x]$-module of rank $1$ with generator 
\[
\tau_{S^1}(W) \in H^r_{S^1}( S_{W} , B_W ; \mathbb{Z}),
\]
the equivariant Thom class of $W$. Let $i_W : B \to W$ denote the zero section. Then $i^*_W\tau_{S^1}(W) = e_{S^1}(W)$, where 
\[
e_{S^1}(W) \in H^r(B ; \mathbb{Z})[x]
\]
is the equivariant Euler class. Similarly, if $W$ is any complex vector bundle of complex rank $r$ equipped with the action of $S^1$ by fibrewise scalar multiplication, the the equivariant cohomology $H^*_{S^1}( S_W , B_W ; \mathbb{Z})$ is a free $H^*(B ; \mathbb{Z})[x]$-module of rank $1$ with generator $\tau_{S^1}(W) \in H^{2r}_{S^1}( S_W , B_W ; \mathbb{Z})$. If $i_W : B \to W$ denotes the zero section, we have $i^*_W\tau_{S^1}(W) = e_{S^1}(W)$, where $e_{S^1}(W) \in H^{2r}(B ; \mathbb{Z})[x]$ is the equivariant Euler class.\\

Let us temporarily assume that $H^+(X)$ is orientable and fix an orientation. Let $\tau_{S^1}(V \oplus U), \tau_{S^1}(V' \oplus U')$ denote the Thom classes of $V \oplus U$ and $V' \oplus U'$. Then:
\begin{equation*}
\begin{aligned}
i^*\tau_{S^1}(V \oplus U) &= e_{S^1}(V) \tau_{S^1}(U), \\
j^*\tau_{S^1}(V' \oplus U') &= e_{S^1}(V') \tau_{S^1}(U'), \\
(f|_{S_U})^*\tau_{S^1}(U') &= e_{S^1}(H^+(X)) \tau_{S^1}(U).
\end{aligned}
\end{equation*}
On the other hand, since $H^*_{S^1}( S_{V , U} , B_{V,U} ; \mathbb{Z})$ and $H^*_{S^1}( S_{V' , U'} , B_{V',U'} ; \mathbb{Z})$ are free $H^*(B ; \mathbb{Z})[x]$-modules generated by $\tau_{S^1}(V \oplus U)$ and $\tau_{S^1}(V'\oplus U')$, we must have
\[
f^* \tau_{S^1}(V' \oplus U') = \beta \tau_{S^1}(V \oplus U)
\]
for some $\beta \in H^{b_+(X)-2d}(B ; \mathbb{Z})[x]$. Applying $i^*$, we obtain:
\begin{equation*}
\begin{aligned}
\beta e_{S^1}(V) \tau_{S^1}(U) &= i^*( \beta \tau_{S^1}(V \oplus U)) \\
&= i^* f^*( \tau_{S^1}(V' \oplus U')) \\
&= (f|_{S_U})^* j^*(\tau_{S^1}(V' \oplus U')) \\
&= (f|_{S_U})^* ( e_{S^1}(V') \tau_{S^1}(U')) \\
&= e_{S^1}(V') e_{S^1}(H^+(X)) \tau_{S^1}(U).
\end{aligned}
\end{equation*}
Equating multiples of $\tau_{S^1}(U)$, we have shown that:
\begin{equation}\label{equ:euler1}
e_{S^1}(V') e_{S^1}(H^+(X)) = \beta e_{S^1}(V)
\end{equation}
for some $\beta \in H^{b_+(X)-2d}(B ; \mathbb{Z})[x]$. Suppose that $C$ is a complex vector bundle of rank $r$. Let $S^1$ act on $C$ fibrewise by scalar multiplication. Then the equivariant Euler class is given by:
\[
e_{S^1}(C) = x^r + x^{r-1}c_1(C) + \cdots + c_r(C).
\]
This follows from the splitting principle (i.e. by pulling back to the flag bundle associated to $C$). On the other hand if $W$ is a real oriented vector bundle of rank $r$, equipped with the trivial $S^1$-action, then $e_{S^1}(W)$ is just the usual non-equivariant Euler class $e(W)$ pulled back to equivariant cohomology. Applying these remarks to Equation (\ref{equ:euler1}), we have:
\[
e(H^+(X)) \left( x^{a'} + x^{a'-1}c_1(V') + \cdots + c_{a'}(V') \right) = \beta  x^a 
\]
for some $\beta \in H^{b_+(X)-2d}(B ; \mathbb{Z})[x]$.\\

In the general case where $H^+(X)$ is not necessarily oriented, we still have equivariant Thom classes and Euler classes provided we use local coefficients. In particular, the Euler class of $H^+(X)$ is now a class $e(H^+(X)) \in H^{b_+(X)}(B , \mathbb{Z}_w)$, where $w = w_1(H^+(X))$ and $\mathbb{Z}_w$ is the local system with coefficient group $\mathbb{Z}$ determined by $w$. The above discussion carries over to this case and we obtain:
\begin{equation}\label{equ:euler2}
e(H^+(X)) \left( x^{a'} + x^{a'-1}c_1(V') + \cdots + c_{a'}(V') \right) = \beta  x^a 
\end{equation}
for some $\beta \in H^{b_+(X)-2d}(B ; \mathbb{Z}_w)[x]$.\\

For any complex virtual vector bundle $V \to B$, and any integer $j \ge 0$, we define the {\em $j$-th Segre class of $V$} \cite[\textsection 3.2]{fult} to be the cohomology class $s_j(V) \in H^{2j}(B ; \mathbb{Z})$ given by $s_j(V) = c_j(-V)$. Equivalently, letting $c(V) = 1 + c_1(V) + c_2(V) + \cdots $ denote the total Chern class of $V$ and $s(V) = s_0(V) + s_1(V) + s_2(V) + \cdots $ the total Segre class, we have that $s(V)$ is uniquely determined by the equation $c(V)s(V) = 1$ in $H^*(B ; \mathbb{Z})$.

Now recall that $D \in K^0(B)$ is the virtual vector bundle $D = V - V'$ and we have assumed that $V = \mathbb{C}^a$ is a trivial bundle. It follows that the Chern classes of $V'$ are the Segre classes of $D$:
\[
c_j(V') = s_j(D).
\]

\begin{theorem}\label{thm:euler}
Let $(E , \mathfrak{s}_{E/B})$ be a spin$^c$-family over $B$ with fibre $(X , \mathfrak{s})$. Then if $e(H^+(X)) \neq 0$, we have $d \le 0$. Moreover, $e(H^+(X))s_j(D)=0$ whenever $j > -d$.
\end{theorem}
\begin{proof}
By the above remarks on Segre classes, Equation (\ref{equ:euler2}) can be re-written as:
\[
e(H^+(X)) \left( x^{a'} + x^{a'-1}s_1(D) + \cdots + s_{a'}(D) \right) = \beta  x^a 
\]
Suppose that $e(H^+(X)) \neq 0$. Then Equation (\ref{equ:euler2}) implies that $\beta \neq 0$. Recall that $\beta \in H^{b_+(X)-2d}(B ; \mathbb{Z}_w)[x]$. Let us expand $\beta$ as:
\[
\beta_{m} x^m + \beta_{m-1} x^{m-1} + \cdots + \beta_0,
\]
where $\beta_j \in H^{b_+(X)-2d-2j}(B ; \mathbb{Z}_w)$ and $\beta_m \neq 0$. Substituting for $\beta$, we have:
\[
e(H^+(X)) \left( x^{a'} + x^{a'-1}s_1(D) + \cdots + s_{a'}(D) \right) = \left( \beta_{m} x^m + \beta_{m-1} x^{m-1} + \cdots + \beta_0 \right) x^a.
\]
Since $e(H^+(X)) \neq 0$, the leading power of $x$ on the left hand side is $x^{a'}$. Similarly, since $\beta_m \neq 0$, the leading power of $x$ on the right hand side is $x^{m+a}$. Equating these gives $a' = m+a$, hence $d = a-a' = -m \le 0$.

Now suppose that $j > -d$, so that $a'-j < a$. Equating coefficients of $x^{a'-j}$, we get $e(H^+(X))s_j(D) = 0$, because the right hand side is a multiple of $x^a$.
\end{proof}

\section{K-theoretic constraints}\label{sec:ktheory}

In this section we repeat the arguments of the previous section, using equivariant complex $K$-theory instead of equivariant cohomology. For this we need to recall the $K$-theoretic Thom and Euler classes.\\

Recall that $K^*_{S^1}(pt) \cong \mathbb{Z}[\xi , \xi^{-1}]$, where $\xi \in K^0_{S^1}(pt)$ is the $K$-theory class of the universal line bundle $ES^1 \times_{S^1} \mathbb{C} \to BS^1$. Then since $S^1$ acts trivially on $B$, we have $K^*_{S^1}(B) \cong K^*(B )[\xi,\xi^{-1}]$. Let $W \to B$ be a real oriented vector bundle of rank $r$ equipped with a spin$^c$-structure and give $W$ the trivial circle action. Then the equivariant $K$-theory $K^*_{S^1}( S_W , B_W )$ is a free $H^*(B)[\xi,\xi^{-1}]$-module of rank $1$ with generator 
\[
\tau^K_{S^1}(W) \in K^r_{S^1}( S_W , B_W ),
\]
the equivariant $K$-theoretic Thom class of $W$. Let $i_W : B \to W$ denote the zero section. Then $i_W^*\tau_{S^1}^K(W) = e^K_{S^1}(W)$, where 
\[
e^K_{S^1}(W) \in K^r(B ; \mathbb{Z})[\xi,\xi^{-1}]
\]
is the equivariant $K$-theoretic Euler class. Similar statements hold in the case of a complex vector bundle $W$ equipped the action of $S^1$ by fibrewise scalar multiplication. Note that in this case $W$ is equipped with a canonical spin$^c$-structure arising from the complex structure.\\

For a real vector bundle which is not spin$^c$, we can still define $K$-theoretic Thom and Euler classes, provided we work with twisted $K$-theory groups. For simplicity, we will avoid this situation and work only with untwisted $K$-theory. As we will see below, in the case that $H^+(X)$ is not spin$^c$ we can still extract information by a doubling trick.\\

Let us first assume that $H^+(X)$ can be given a spin$^c$-structure and fix such a choice. The arguments used in Section \ref{sec:cohomology} directly carry over to $K$-theory and the analogue of Equation (\ref{equ:euler1}) is:
\begin{equation}\label{equ:eulerk1}
e^K_{S^1}(V') e^K(H^+(X)) = \alpha e^K_{S^1}(V),
\end{equation}
for some $\alpha \in K^{b_+(X)}(B)[\xi , \xi^{-1}]$. Next we recall that for a complex vector bundle $W$ of rank $r$, the $K$-theoretic Euler class is given by\footnote{Our convention for $K$-theoretic Thom classes is such that the corresponding $K$-theoretic Euler class is given $\Lambda_{-1}(W^*)$. There is another commonly used convention in which the $K$-theoretic Euler class would be $\Lambda_{-1}(W)$.} 
\[
e^K_{S^1}(W) = \Lambda_{-1}(W^*) = \sum_{i=0}^r (-1)^i \wedge^i W^*.
\]
In our case $V \in K^0(B)$ is a complex vector bundle on $B$ which is made $S^1$-equivariant by letting $S^1$ act by scalar multiplication. In terms of equivariant $K$-theory this simply means that we should replace $V$ by $V \otimes \xi \in K^0_{S^1}(B)$ and similarly replace $V'$ by $V' \otimes \xi$. Substituting into Equation (\ref{equ:eulerk1}) and assuming $V = \mathbb{C}^a$ is a trivial bundle gives:
\begin{equation}\label{equ:eulerk2}
e^K(H^+(X)) \left( \sum_{i=0}^{a'} (-1)^i \wedge^i V'^* \otimes \xi^{-i} \right) = \alpha (1-\xi^{-1})^{a}
\end{equation}
for some $\alpha \in K^{b_+(X)}(B)[\xi , \xi^{-1}]$. 

\begin{proposition}
Let $(E , \mathfrak{s}_{E/B})$ be a spin$^c$-family over $B$ with fibre $(X , \mathfrak{s})$, where $b_+(X) = 0$ and $d=0$. Then $[D] = 0$ in $K^0(B)$.
\end{proposition}
\begin{proof}
If $b_+(X) = 0$, then $H^+(X) = 0$, which is certainly spin$^c$. So Equation (\ref{equ:eulerk2}) implies that:
\[
\left( \sum_{i=0}^{a'} (-1)^i \wedge^i V'^* \otimes \xi^{-i} \right) = \alpha (1-\xi^{-1})^{a}
\]
for some $\alpha \in K^{b_+(X)}(B)[\xi , \xi^{-1}]$. Suppose further that $d = 0$, so $a = a'$. Then by considering the highest and lowest powers of $\xi$, we see that the above equation is only possible if $\alpha = 1$. But looking at coefficients of $\xi^{-1}$, this implies that $[V'^*] = \mathbb{C}^{a}$. So $[V] = [V'] = \mathbb{C}^{a}$ and hence $[D] = [V] - [V'] = 0$.
\end{proof}

In the general setting where $H^+(X)$ need not admit a spin$^c$-structure, we could attempt to carry out the above construction using twisted $K$-theory. Instead we will follow a different approach. Consider the fibrewise smash product of $f$ by itself:
\[
f \wedge_B f: S_{V \oplus V, U \oplus U} \to S_{V' \oplus V' , U' \oplus U'}
\]
Then $f \wedge_B f$ is also an $S^1$-equivariant map between sphere bundles and we can study it in exactly the same way that we have been doing for $f$. The effect of this doubling trick is to replace $V,V',U,U'$ by their doubles $V \oplus V, V' \oplus V', U \oplus U, U' \oplus U'$ and hence also to replace $D$ by $D \oplus D$ and $H^+(X)$ by $H^+(X) \oplus H^+(X)$. Then since $H^+(X)_{\mathbb{C}} = H^+(X) \oplus H^+(X)$ has a natural complex structure, it also has an induced spin$^c$-structure and so we can apply the same reasoning used above to deduce that:
\begin{equation}\label{equ:eulerk3}
e^K(H^+_{\mathbb{C}}(X)) e^K_{S^1}(V')^2 = \gamma (1-\xi^{-1})^{2a}
\end{equation}
for some $\gamma \in K^0(B)[\xi, \xi^{-1}]$. Note that $H^+(X)_{\mathbb{C}} \cong (H^+(X)_{\mathbb{C}})^*$ because $H^+(X)_{\mathbb{C}}$ is the complexification of a real bundle.

\begin{theorem}\label{thm:eulerk}
Let $(E , \mathfrak{s}_{E/B})$ be a spin$^c$-family over $B$ with fibre $(X , \mathfrak{s})$. Then if $e^K(H^+(X)_{\mathbb{C}}) \neq 0$, we have $d \le 0$.
\end{theorem}
\begin{proof}
Equation (\ref{equ:eulerk3}) gives:
\begin{equation}\label{equ:eulerk4}
e^K(H^+_{\mathbb{C}}(X)) \left( \sum_{i=0}^{a'} (-1)^i \wedge^i V'^* \otimes \xi^{-i} \right)^2 = \gamma (1-\xi^{-1})^{2a}
\end{equation}
for some $\gamma \in K^0(B)[\xi,\xi^{-1}]$. If $e^K(H^+(X)_{\mathbb{C}}) \neq 0$, then by considering coefficients of $\xi^0$, we see that $\gamma \neq 0$. Thus there exists integers $m \ge n$ such that
\[
\gamma = \gamma_m \xi^m + \gamma_{m-1} \xi^{m-1} + \cdots + \gamma_n \xi^n
\]
and where $\gamma_m, \gamma_n \neq 0$. Note that $\wedge^{a'} V'^*$ is the $K$-theory class of a line bundle, hence is an invertible element of $K^0(B)$. Therefore $e^K(H^+(X)_{\mathbb{C}}) \neq 0$ implies also that $e^K(H^+(X)_{\mathbb{C}}) \wedge^{a'} V'^* \neq 0$. Comparing highest and lowest powers of $\xi$ appearing in the left and right hand sides of (\ref{equ:eulerk4}), we find that $m =0$ and $2a' = 2a+n$. Hence $-2d = n \le m = 0$ and $d \le 0$.
\end{proof}

\section{Monopole map for spin families}\label{sec:spin}

In this section we consider the monopole map for a spin family $(E , \mathfrak{s}_{E/B})$ with fibre $(X , \mathfrak{s})$ over a base $B$. Recall that it is a $Pin(2)$-equivariant map
\[
f : S_{V , U} \to S_{V' , U'},
\]
where $V,V'$ are quaternionic vector bundles of ranks $2a,2a'$, $U,U'$ are real bundles of ranks $b,b'$, $Pin(2)$ acts fibrewise on $V,V'$ through the quaternionic structures, $j$ acts on $U,U'$ as multiplication by $-1$ and $S^1$ acts trivially on $U,U'$. As usual we also assume that $V,U$ are trivial.\\

We first consider $Pin(2)$ equivariant cohomology. Following \cite{man}, it is useful to view $Pin(2)$ as a subgroup of $Sp(1) = SU(2)$. Identifying $SU(2)$ with the unit $3$-sphere $S^3 \subset \mathbb{C}^2$, we see that $SU(2)/S^1 = \mathbb{CP}^1$. One then finds that $j$ acts on $\mathbb{CP}^1$ as the antipodal map so that $SU(2)/Pin(2) = \mathbb{RP}^2$. Next, we view $BSU(2) = ESU(2)/SU(2)$ as $\mathbb{HP}^\infty$. We can then identify $BPin(2)$ with $ESU(2)/Pin(2)$ and we obtain a fibration
\[
\mathbb{RP}^2 \to BPin(2) \to \mathbb{HP}^\infty.
\]
Recall that $H^*( \mathbb{HP}^\infty ; \mathbb{Z}) = \mathbb{Z}[v]$, where $v = -c_2(V)$ and $V \to BSU(2) = \mathbb{HP}^\infty$ is the universal $SU(2)$ bundle. The Leray-Serre spectral sequence immediately implies that 
\[
H^*( BPin(2) ; \mathbb{Z}) = \mathbb{Z}[v,w]/\langle 2w , w^2 \rangle, \quad deg(w) = 2, \; \; deg(v) = 4
\]
and
\[
H^*( BPin(2) ; \mathbb{Z}_2) = \mathbb{Z}_2[v,u]/\langle  u^3 \rangle, \quad deg(u) = 1, \; \; deg(v) = 4.
\]
Assume temporarily that $U,U'$ are equivariantly orientable, that is, assume $U,U'$ are orientable and that the action of $j$ preserves orientations. Then by the Thom isomorphism $H^*_{Pin(2)}( S_{V , U} , B_{V,U} ; \mathbb{Z})$ and $H^*_{Pin(2)}( S_{V' , U'} , B_{V',U'} ; \mathbb{Z})$ are free $H^*_{Pin(2)}(B ; \mathbb{Z})$-modules generated by $\tau_{Pin(2)}(V \oplus U)$ and $\tau_{Pin(2)}(V'\oplus U')$. We then have:
\[
f^*\tau_{Pin(2)}(V' \oplus U') = \beta \tau_{Pin(2)}(V \oplus U)
\]
for some $\beta \in H^{b_+(X)-2d}_{Pin(2)}(B ; \mathbb{Z})$. Arguing as we did in Section \ref{sec:cohomology}, we find that:
\begin{equation}\label{equ:eulerpin1}
e_{Pin(2)}(V') e_{Pin(2)}(H^+(X)) = \beta e_{Pin(2)}(V)
\end{equation}
for some $\beta \in H^{b_+(X)-2d}_{Pin(2)}(B ; \mathbb{Z})$. Here $e_{Pin(2)}(W)$ denotes the $Pin(2)$-equivariant Euler class of a vector bundle $W$. Suppose that $W$ is a quaternionic vector bundle of complex rank $2r$. Let $Pin(2)$ acts fibrewise through the quaternionic structure. Then the $Pin(2)$-equivariant Euler class is the restriction to $Pin(2)$ of the corresponding $Sp(1)$-equivariant Euler class $e_{Sp(1)}(W) \in H^*_{Sp(1)}(B ; \mathbb{Z}) = H^*(B ; \mathbb{Z})[v]$. Pulling back to $S^1$-equivariant cohomology, we get a map $H^*(B ; \mathbb{Z})[v] \to H^*(B ; \mathbb{Z})[x]$ which sends $v$ to $x^2$. To see this note that the universal $SU(2)$ bundle pulled back to $BS^1 = \mathbb{CP}^\infty$ becomes $\mathcal{O}(1) \oplus \mathcal{O}(-1)$, so $v$ pulls back to $-c_2( \mathcal{O}(1) \oplus \mathcal{O}(-1) ) = x^2$. But since $e_{Sp(1)}(W)$ pulls back to $e_{S^1}(W)$, it follows that
\[
e_{Sp(1)}(W) = v^r + c_2(W)v^{r-1} + \cdots + c_{2r}(W).
\]
Next, consider the homomorphism $Pin(2) \to \mathbb{Z}_2$ which sends $j$ to $-1$ and sends $S^1$ to the identity. Let $W$ be a real vector bundle of rank $r$ and let $Pin(2)$ acts on $W$ by letting $j$ acts as $-1$ on the fibres and letting $S^1$ act trivially. Then $e_{Pin(2)}(W)$ is the pullback of $e_{\mathbb{Z}_2}(W)$ to $Pin(2)$-equivariant cohomology. We can re-write Equation (\ref{equ:eulerpin1}) as:
\begin{equation}\label{equ:eulerpin2}
e_{\mathbb{Z}_2}(H^+(X)) ( v^{a'} + s_2(D)v^{a'-1} + \cdots + s_{2a'}(D) ) = \beta v^a
\end{equation}
for some $\beta \in H^{b_+(X)-2d}_{Pin(2)}(B ; \mathbb{Z})$. Note that in (\ref{equ:eulerpin2}) $e_{\mathbb{Z}_2}(H^+(X))$ really means the pullback of $e_{\mathbb{Z}_2}(H^+(X))$ to $Pin(2)$-equivariant cohomology. This class does not contain any positive power of $v$ because the image of $H^*( B\mathbb{Z}_2 ; \mathbb{Z}) \to H^*( BPin(2) ; \mathbb{Z})$ doesn't contain any positive power of $v$.\\

In the general case where $U,U'$ need not be equivariantly orientable, the above argument works provided we use local coefficients. The equivariant Euler class of $H^+(X)$ is now a class $e_{\mathbb{Z}_2}(H^+(X)) \in H^{b_+(X)}_{\mathbb{Z}_2}( B ; \mathbb{Z}_w )$, where $w = w_{1,\mathbb{Z}_2}( H^+(X)) \in H^1_{\mathbb{Z}_2}( B ; \mathbb{Z}_2)$ is the equivariant first Stiefel-Whitney class of $H^+(X)$ and $\mathbb{Z}_w$ is the corresponding local system on $B \times B\mathbb{Z}_2$ with coefficient group $\mathbb{Z}$ determined by $w$. Then Equation (\ref{equ:eulerpin1}) still holds for some $\beta \in H^{b_+(X)-2d}_{Pin(2)}(B ; \mathbb{Z}_w)$. Arguing exactly as in the proof of Theorem \ref{thm:euler}, we obtain:

\begin{theorem}\label{thm:eulerpin}
Let $(E , \mathfrak{s}_{E/B})$ be a spin family over $B$ with fibre $(X , \mathfrak{s})$. Then if the image of $e_{\mathbb{Z}_2}(H^+(X))$ in $H^*_{Pin(2)}(B ; \mathbb{Z}_w)$ is non-zero, we have $d \le 0$.
\end{theorem}
To use this theorem effectively, we need to be able to compute $e_{\mathbb{Z}_2}(H^+(X)) \in H^{b_+(X)}_{\mathbb{Z}_2}(B ; \mathbb{Z}_w)$, or at least determine when it is non-zero. Note that upon reduction to $\mathbb{Z}_2$ coefficients $e_{\mathbb{Z}_2}(H^+(X))$ reduces to $w_{b_+(X) , \mathbb{Z}_2}(H^+(X)) \in H^{b_+(X)}_{\mathbb{Z}_2}(B ; \mathbb{Z}_2 )$, the top equivariant Stiefel-Whitney class of $H^+(X)$. We observe that
\[
H^*_{\mathbb{Z}_2}(B ; \mathbb{Z}_2) = H^*(B ; \mathbb{Z}_2)[u],
\]
where $u = w_1( \mathcal{O}_{\mathbb{R}}(1) ) \in H^1_{\mathbb{Z}_2}( pt ; \mathbb{Z}_2)$ is the first Stiefel-Whitney class of the line bundle $\mathcal{O}_{\mathbb{R}}(1) \to \mathbb{RP}^\infty$. Recall that $H^+(X)$ is made into a $\mathbb{Z}_2$-equivariant line bundle by having the generator of $\mathbb{Z}_2$ act fibrewise by $-1$. It follows that $w_{b_+(X) , \mathbb{Z}_2}(H^+(X))$ coincides with $w_{b_+(X)}( H^+(X) \otimes \mathcal{O}_{\mathbb{R}}(1) )$ under the isomorphism $H^*_{\mathbb{Z}_2}(B ; \mathbb{Z}_2) \cong H^*(B \times \mathbb{RP}^\infty ; \mathbb{Z}_2)$. From this we find:
\[
w_{b_+(X)}(H^+(X)) = u^{b_+(X)} + w_1(H^+(X))u^{b_+(X)-1} + \cdots + w_{b_+(X)}(H^+(X)).
\]
We must consider whether this class is non-zero after pulling back to $Pin(2)$-equivariant cohomology. Recall that $u^3 =0$ in $H^*_{Pin(2)}(pt ; \mathbb{Z}_2)$. Thus:
\[
w_{b_+(X) , Pin(2)}(H^+(X)) = w_{b_+(X)}(H^+(X)) + u w_{b_+(X)-1}(H^+(X)) + u^2 w_{b_+(X)-2}(H^+(X)).
\]
We thus have:
\begin{corollary}\label{cor:eulerpin}
Let $(E , \mathfrak{s}_{E/B})$ be a spin family over $B$ with fibre $(X , \mathfrak{s})$. If $w_i( H^+(X) ) \neq 0$ for $i = b_+(X), b_+(X)-1$ or $b_+(X)-2$, then $d \le 0$.
\end{corollary}

\begin{remark}
Note that since $\mathfrak{s}$ is a spin structure, we have $d = -\sigma(X)$. When $B = pt$, Corollary \ref{cor:eulerpin} gives that if $X$ is a smooth spin $4$-manifold with $b_1(X)=0$ and $b^+(X) \le 2$, then $\sigma(X) \ge 0$. This shows for example that the topological $4$-manifold $\#^2(S^2 \times S^2) \#^2 \overline{E}_8$ has no smooth structure.
\end{remark}

Now we consider $Pin(2)$-equivariant $K$-theory. To avoid issues of $K$-orientability, we repeat the trick used in Section \ref{sec:ktheory} and replace $f$ with its double. Actually we need to modify the construction a little to account for the $Pin(2)$-symmetry. Let
\[
\overline{f} : S_{\overline{V} , U}  \to S_{\overline{V'} , U'}
\]
denote the map $f$, but where we use the complex conjugate circle action on the domain and target. Now we consider the ``complexification" of $f$, namely the map
\[
f_{\mathbb{C}} = f \wedge_B \overline{f} : S_{V \oplus \overline{V} , U \oplus U} \to S_{V' \oplus \overline{V'} , U' \oplus U'}.
\]
Note that we can identify the complex vector bundle $V \oplus \overline{V}$ with the complexification $V \otimes_{\mathbb{R}} \mathbb{C}$, where the complex structure $i$ on $\mathbb{C}$ corresponds to $i = diag( I , -I)$ on $V \oplus \overline{V}$.

Now let $J : V \to V$ denote the quaternionic structure on $V$. Then $V \otimes_{\mathbb{R}} \mathbb{C}$ admits an action of the quaternions commuting with the complex structure $i$. Under the identification $V \otimes_{\mathbb{R}} \mathbb{C} \cong V \oplus \overline{V}$, we can define the action of the quaternions on $V \oplus \overline{V}$ to be given by $\hat{I}, \hat{J}, \hat{K}$, where
\[
\hat{I} = \left[ \begin{matrix} I & 0 \\ 0 & I \end{matrix} \right], \quad \hat{J} = \left[ \begin{matrix} 0 & J \\ J & 0 \end{matrix} \right], \quad \hat{K} = \left[ \begin{matrix} 0 & IJ \\ IJ & 0 \end{matrix} \right].
\]
The point of this exercise is that we obtain a quaternionic structure on $V_{\mathbb{C}} = V \oplus \overline{V}$ commuting with the natural complex structure $i = diag( I , -I)$ and hence an action of $Sp(1)$ by {\em complex linear} isomorphisms. A similar remark holds for $V' \oplus \overline{V'}$. Moreover the complexified map $f_{\mathbb{C}}$ is equivariant under the action of $Pin(2) = \{ e^{\hat{I}\theta} \} \cup \{ \hat{J} e^{\hat{I}\theta} \}$. Now by an argument similar to that used in Section \ref{sec:ktheory}, we obtain:
\begin{equation}\label{equ:gammak}
e^K_{\mathbb{Z}_2}(H^+(X)_{\mathbb{C}}) e^K_{Sp(1)}(V'_{\mathbb{C}}) = \gamma e^K_{Sp(1)}(V_{\mathbb{C}})
\end{equation}
for some $\gamma \in K^0_{Pin(2)}(B)$. In the above equation, we pull back classes in $\mathbb{Z}_2$- and $Sp(1)$-equivariant cohomology by the natural maps $Pin(2) \to \mathbb{Z}_2$ and $Pin(2) \to Sp(1)$.\\

As abelian groups $K^0_{\mathbb{Z}_2}(pt) = R[\mathbb{Z}_2]$ is free abelian with generators $1,1_-$, where $1$ denotes the trivial representation and $1_-$ the sign representation. $K^0_{Pin(2)}(pt) = R[Pin(2)]$ is free abelian with generators $1,1_-, \mu_j$, $j \ge 1$, where $\mu_j$ restricts to $\xi^j + \xi^{-j}$ in $R[S^1]$. Recall that we make $H^+(X)_{\mathbb{C}}$ into a $Pin(2)$-equivariant bundle by letting $j$ act as $-1$. It follows that:
\[
e^K_{\mathbb{Z}_2}(H^+(X)_{\mathbb{C}}) = \sum_{i \; even} \wedge^i H^+(X)_{\mathbb{C}} - \sum_{i \; odd} \wedge^i H^+(X)_{\mathbb{C}} \otimes 1_-.
\]
Next, let $W$ be any complex vector bundle with an action of the quaternions by complex linear isomorphisms. Since $I^2 = -1$, we can decompose $W$ into the $\pm i$ eigenspaces of $I$. Let $W_0$ denote the $+i$-eigenspace. Then since $J$ anti-commutes with $I$ it exchanges the $\pm i$ eigenspaces isomorphically. It follows that we can identify $W$ with $W_0 \oplus W_0$ and $I,J$ with 
\[
I = \left[ \begin{matrix} i & 0 \\ 0 & -i \end{matrix} \right], \quad J = \left[ \begin{matrix} 0 & -1 \\ 1 & 0 \end{matrix} \right].
\]
It follows that $W = \mu_1 \otimes_{\mathbb{C}} W_0$ as complex vector bundles equipped with actions of the quaternions. In particular, if $W$ is of the form $W = V_{\mathbb{C}} = V \otimes_{\mathbb{R}} \mathbb{C}$, then $W = \mu_1 \otimes_{\mathbb{C}} V$. Therefore

\[
e^K_{Pin(2)}(V_{\mathbb{C}}) = \sum_{i \ge 0} (-1)^i \wedge^i ( \mu_1 \otimes V).
\]
To compute this class we will use the splitting principle in $K$-theory. If $V$ is a sum $V = \oplus_{i=1}^{a} V_i$ of line bundles then:
\begin{equation*}
\begin{aligned}
e^K_{Pin(2)}(\mu_1 \otimes V) &= \prod_{i=1}^a (1 - \mu_1 \otimes V_i^{-1} + \wedge^2( \mu_1 \otimes V_i^{-1}) ) \\
&= \prod_{i=1}^a ( 1 - \mu_1 \otimes V_i^{-1} + V_i^{-2} ).
\end{aligned}
\end{equation*}
Now we consider the homomorphism $tr_j : K^0_{Pin(2)}(B) \to K^0(B)$ given by evaluating the character of $Pin(2)$ representations at $j$. Under this map the trivial representation is sent to $1$, the representation $1_-$ is sent to $-1$ and $\mu_k$ is sent to $0$ for all $k \ge 1$. Therefore we have:
\begin{equation*}
\begin{aligned}
tr_j( e^K_{\mathbb{Z}_2}(H^+(X)_{\mathbb{C}}) ) &= \sum_{i \; even} \wedge^i H^+(X)_{\mathbb{C}} + \sum_{i \; odd} \wedge^i H^+(X)_{\mathbb{C}} = \wedge^* H^+(X)_{\mathbb{C}}, \\
tr_j(e^K_{Pin(2)}(V_{\mathbb{C}} )) &= \prod_{i=1}^a (1 + V_i^{-2}) = \wedge^* \psi^2(V^*) = \wedge^* \psi^2(V),
\end{aligned}
\end{equation*}
where $\psi^2$ is the second Adams operator and the last equality follows from $V \cong V^*$ (since $V$ has a quaternionic structure). Similarly $tr_j( e^K_{Pin(2)}( V'_{\mathbb{C}})) = \wedge^* \psi^2(V'^*)$. Putting all this together, we have shown:
\begin{theorem}\label{thm:10on8}
Let $(E , \mathfrak{s}_{E/B})$ be a spin family over $B$ with fibre $(X , \mathfrak{s})$. Then
\[
\wedge^* H^+(X)_{\mathbb{C}} \otimes \wedge^* \psi^2(V') =  \eta  (\wedge^* \psi^2(V))
\]
for some $\eta \in K^0(B)$.
\end{theorem}

Theorem \ref{thm:10on8} can be improved by a factor of $2$ if the $K$-theoretic Euler class of $H^+(X)_\mathbb{C}$ vanishes and we are willing to sacrifice torsion:
\begin{theorem}\label{thm:10on8plus1}
Let $(E , \mathfrak{s}_{E/B})$ be a spin family over $B$ with fibre $(X , \mathfrak{s})$. Assume $e^K(H^+(X)_{\mathbb{C}})=0$. Then in $K^0(B)/torsion$ we have:
\[
\wedge^* H^+(X)_{\mathbb{C}} \otimes \wedge^* \psi^2(V') =  2 \eta  (\wedge^* \psi^2(V))
\]
for some $\eta \in K^0(B)$.
\end{theorem}
\begin{proof}
Our proof is similar to arguments used in \cite{bry}. Starting from Equation (\ref{equ:gammak}), we have
\begin{equation}\label{equ:gammak2}
\left( \sum_{i \; even} \wedge^i H^+(X)_{\mathbb{C}} - \sum_{i \; odd} \wedge^i H^+(X)_{\mathbb{C}} \otimes 1_- \right) e^K( \mu_1 \otimes V') = \gamma e^K(\mu_1 \otimes V)
\end{equation}
for some $\gamma \in K^0_{Pin(2)}(B)$. Let $K^0(B)_\mathbb{C} = K^0(B) \otimes_{\mathbb{Z}} \mathbb{C}$. Let $\xi \in S^1$ and consider the homomorphism $tr_\xi : K^0_{Pin(2)}(B) \to K^0(B)_{\mathbb{C}}$ given by evaluating the character of $Pin(2)$ representations at $\xi$. We have $tr_\xi(\mu_k) = \xi^k + \xi^{-k}$ and $tr_\xi(1) = tr_\xi(1_-) = 1$. Using the splitting principle to write $V = \bigoplus_{i=1}^a V_i$, we find
\[
tr_\xi( e^K(\mu_1 \otimes V) ) = \prod_{i=1}^{a} ( 1 -(\xi+\xi^{-1})V_i^{-1} + V_i^{-2}).
\]
Thus if we apply $tr_\xi$ to (\ref{equ:gammak2}), we get:
\begin{equation}\label{equ:gammak3}
0 = e^K(H^+(X)_\mathbb{C}) = tr_\xi(\gamma) \prod_{i=1}^{a} ( 1 -(\xi+\xi^{-1})V_i^{-1} + V_i^{-2}) \in K^0(B)_{\mathbb{C}}.
\end{equation}
We claim this implies that $tr_\xi(\gamma) = 0$ for all $\xi \in S^1$. To see this we first note that
\[
\prod_{i=1}^{a} ( 1 -(\xi+\xi^{-1})V_i^{-1} + V_i^{-2}) = (-1)^a det(V)^{-1} \xi^{-a} \prod_{i=1}^a (\xi - V_i^{-1})(\xi - V_i)
\]
and $\prod_{i=1}^{a}(\xi-V_i^{-1})(\xi - V_i)$ can be written as a monic degree $2a$ polynomial in $\xi$ with coefficients in $K^0(B)_{\mathbb{C}}$:
\[
\prod_{i=1}^{a}(\xi-V_i^{-1})(\xi - V_i) = \xi^{2a} + c_1 \xi^{2a-1} + \cdots + c_{2a}, \quad c_1, \dots , c_{2a} \in K^0(B)_{\mathbb{C}}.
\]
Further, we may write $\gamma$ as $\gamma = \gamma_0 + \widetilde{\gamma_0} 1_- + \sum_{i \ge 1} \gamma_i \mu_i$ for some $\gamma_0, \widetilde{\gamma}_0 , \gamma_i \in K^0(B)$, where only finitely many of the $\gamma_i$ are non-zero. Then
\[
tr_\xi(\gamma) = (\gamma_0 - \widetilde{\gamma}_0) + \sum_{i \ge 1} \gamma_i( \xi^i + \xi^{-i}).
\]
Since only finitely many of the $\gamma_i$ are non-zero, there exists an $m \ge 0$ such that $y(\xi) = \xi^m tr_\xi(\gamma)$ is a polynomial in $\xi$ with coefficients in $K^0(B)_{\mathbb{C}}$. Suppose that $tr_\xi(\gamma) \neq 0$. Then $y(\xi)$ is a non-zero polynomial. Let $r \ge 0$ be the degree of $y(\xi)$, so $y(\xi) = y_r \xi^r + \cdots + y_0$ for some $y_0, \dots , y_r \in K^0(B)_\mathbb{C}$ with $y_r \neq 0$. From (\ref{equ:gammak3}) we get that $y(\xi)( \xi^{2a} + \cdots + c_{2a}) = 0$ for all $\xi \in S^1$. However if $B$ is a compact finite dimensional manifold, then $K^0(B)_{\mathbb{C}}$ is finite dimensional over $\mathbb{C}$ and hence the polynomial $y(\xi)( \xi^{2a} + \cdots + c_{2a})$ vanishes for all $\xi \in \mathbb{C}$. But if $y(\xi)(\xi^{2a} + \cdots + c_{2a}) = y_r \xi^{2a+r} + \cdots + y_0 c_{2a}$ vanishes for all $\xi$ then (using finite dimensionality of $K^0(B)_\mathbb{C}$) it follows that $y_r = 0$, a contradiction.

It follows that $tr_\xi(\gamma) = 0$ for all $\xi \in S^1$. This can only happen if $\gamma_0 = \widetilde{\gamma_0}$ and $\gamma_i = 0$ for all $i \ge 1$. Hence $\gamma = \eta(1 - 1_-)$, where $\eta = \gamma_0 \in K^0(B)$. Equation (\ref{equ:gammak2}) now becomes
\[
\left( \sum_{i \; even} \wedge^i H^+(X)_{\mathbb{C}} - \sum_{i \; odd} \wedge^i H^+(X)_{\mathbb{C}} \otimes 1_- \right) e^K( \mu_1 \otimes V') = \eta(1-1_-) e^K(\mu_1 \otimes V)
\]
for some $\eta \in K^0(B)$. Applying $tr_j$ and arguing as in the proof of Theorem \ref{thm:10on8} we now obtain:
\[
\wedge^* H^+(X)_{\mathbb{C}} \otimes \wedge^* \psi^2(V') =  2 \eta  (\wedge^* \psi^2(V))
\]
in $K^0(B)_{\mathbb{C}}$. This implies the result since the image of $K^0(B)$ under the map $K^0(B) \to K^0(B)_{\mathbb{C}}$ can be identified with $K^0(B)/torsion$.
\end{proof}

\begin{remark}
If we take $B = \{pt\}$ to be a point and assume $b_+(X)>0$, then Theorem \ref{thm:10on8plus1} reduces to the statement that $2^{b_+(X)+2a'} = \eta 2^{1+2a}$ for some $\eta \in \mathbb{Z}$ (recall that in this section we take $V,V'$ to have ranks $2a,2a'$). Therefore $b_+(X)+2a' \ge 2a+1$, or $b_+(X) \ge d+1$, which is Furuta's $10/8$ inequality.
\end{remark}

\section{$G$-equivariant monopole map}\label{sec:equivariant}

In this section we consider an equivariant monopole map with respect to a finite group $G$ acting on $X$ by orientation preserving diffeomorphisms. We assume that $G$ preserves the isomorphism class of a spin$^c$-structure $\mathfrak{s}$ but does not necessarily lift to a $G$-action on the spinor bundles. The $G$-equivariant Bauer-Furuta invariant was constructed in \cite{szy2} and we refer the reader to \cite{szy2} for the details of the construction.\\

To construct a finite dimensional approximation of the families Seiberg-Witten monopole map one needs to choose a metric and reference spin$^c$-connection. By averaging, we can assume that the metric has been chosen $G$-invariantly. Lifting the action of $G$ to the spinor bundles, we obtain a central extension
\[
1 \to S^1 \to \widehat{G} \to G \to 1
\]
where $S^1$ acts as constant gauge transformations. Note that this is a split extension if and only if the $G$-action can be lifted to a $G$-action on the spinor bundles. In such a case $\widehat{G} \cong S^1 \times G$ and we say that the $G$-action is {\em liftable}.\\

By averaging over $\widehat{G}$, we can assume a reference spin$^c$-connection has been chosen $\widehat{G}$-invariantly. It then follows that the finite dimensional approximation of the Seiberg-Witten monopole map can be constructed $\widehat{G}$-equivariantly. Therefore we obtain a $\widehat{G}$-equivariant map:
\[
f : S_{V , U} \to S_{V' , U'},
\]
where $V,V'$ are complex representations of $\widehat{G}$ and $U,U' $ are real $\widehat{G}$-representations. Moreover the $S^1$ subgroup of $\widehat{G}$ acts in the usual way, namely as scalar multiplication in the fibres of $V,V'$ and trivially on $U,U'$. The following relations hold in $K^0_{\widehat{G}}(pt)$ and $KO^0_{\widehat{G}}(pt)$ respectively:
\[
V - V' = D, \quad \quad U' - U = H^+(X),
\]
where $D \in K^0_{\widehat{G}}(pt)$ is the $\widehat{G}$-equivariant index of $(X , \mathfrak{s})$ and $H^+(X)$ is the space of harmonic self-dual $2$-forms on $X$.\\

The results of the previous sections can be enhanced to the $\widehat{G}$-equivariant setting (and setting $B = \{pt\}$). We will summarise these results below. The analogue of Equation (\ref{equ:euler1}) is:
\begin{equation}\label{equ:Geuler1}
e_{\widehat{G}}(V') e_{\widehat{G}}(H^+(X)) = \beta e_{\widehat{G}}(V)
\end{equation}
for some $\beta \in H^{b_+(X)-2d}_{\widehat{G}}(pt ; \mathbb{Z}_w)$, where $w$ is the equivariant first Stiefel-Whitney class of $H^+(X)$. Here we are using that $H^*_{\widehat{G}}(pt ; \mathbb{Z}_w) \cong H^*_{\widehat{G}}( B\hat{G} ; \mathbb{Z}_w)$ and we obtain the formula by considering the Borel family $X \times_{\hat{G}} E\hat{G}$ over $B\hat{G}$.

\begin{theorem}\label{thm:Geuler}
Let $G$ act smoothly on $X$ preserving the isomorphism class of $\mathfrak{s}$. Suppose that the $G$-action is liftable. Then if $e_G(H^+(X)) \neq 0$, we have $d \le 0$. Moreover $e_G(H^+(X))s_{j,G}(D) = 0$ whenever $j > -d$, where $s_{j,G}(D)$ denotes the $j$-th equivariant Segre class of $G$.
\end{theorem}
\begin{proof}
First note that since $G$ is liftable, we have $\widehat{G} = S^1 \times G$ and $H^*_{\widehat{G}}(pt ; \mathbb{Z}_w) \cong H^*_{G}(pt ; \mathbb{Z}_w)[x]$, where $x$ is the generator of $H^2_{S^1}(pt ; \mathbb{Z})$. Bearing this in mind, the proof of Theorem \ref{thm:euler} is easily seen to adapt to the $G$-equivariant setting.
\end{proof}

Turning to $K$-theory, let us first assume that $H^+(X)$ can be given a $G$-equivariant spin$^c$-structure and fix such a choice. The analogue of Equation (\ref{equ:eulerk1}) is:
\begin{equation}\label{equ:Geulerk1}
e^K_{\widehat{G}}(V') e^K_{\widehat{G}}(H^+(X)) = \alpha e^K_{\widehat{G}}(V),
\end{equation}
for some $\alpha \in K^{b_+(X)}_{\widehat{G}}(pt)$. Without assuming a $G$-equivariant spin$^c$-structure on $H^+(X)$ we can replace $f$ by its ``complexification" $f \wedge \overline{f}$, and obtain:
\begin{equation}\label{equ:Geulerk2}
e^K_{\widehat{G}}(V'_{\mathbb{C}}) e^K_{\widehat{G}}(H^+(X)_{\mathbb{C}}) = \gamma e^K_{\widehat{G}}(V_{\mathbb{C}} ),
\end{equation}
for some $\gamma \in K^0_{\widehat{G}}(pt)$.

\begin{theorem}
Let $G$ act smoothly on $X$ preserving the isomorphism class of $\mathfrak{s}$. Suppose that the $G$-action is liftable. Then if $e_G^K(H^+(X)_{\mathbb{C}}) \neq 0$, we have $d \le 0$.
\end{theorem}
\begin{proof}
Since $G$ is liftable, we have $K^0_{\widehat{G}}(pt) \cong K^0_G(pt) \otimes K^0_{S^1}(pt) = K^0_G(pt)[\xi , \xi^{-1}]$. Then it is fairly straightforward to adapt the proof of Theorem \ref{thm:eulerk} to the $G$-equivariant setting. 
\end{proof}

Now we consider the case where $X$ is given a spin structure whose isomorphism class is preserved by $G$. In this case the group of lifts of $G$ to automorphisms of the spin principal bundle defines a central extension:
\[
0 \to \mathbb{Z}_2 \to \widetilde{G} \to G \to 1.
\]
This is a split extension if and only if the $G$-action can be lifted to a $G$-action on the associated principal $Spin(4)$-bundle. In such a case $\widetilde{G} \cong \mathbb{Z}_2 \times G$ and we say that the $G$-action is {\em spin-liftable}. Next, we define $Pin^G(2) = \widetilde{G} \times_{\mathbb{Z}_2} Pin(2)$, where $\mathbb{Z}_2$ is taken as a subgroup of $Pin(2)$ via $\mathbb{Z}_2 \subset S^1 \subset Pin(2)$. This group acts on the spinor bundles in the obvious way. Note also that if $G$ is spin-liftable then $Pin^G(2) \cong Pin(2) \times G$.\\

Taking as usual a finite dimensional approximation of the monopole map, we obtain a $Pin^G(2)$-equivariant map
\[
f : S_{V , U} \to S_{V' , U'},
\]
where $V,V'$ are quaternionic representations of ranks $2a,2a'$, $U,U'$ are real representations of ranks $b,b'$, $Pin(2)$ acts on $V,V'$ through the quaternionic structures, $j$ acts on $U,U'$ as multiplication by $-1$ and $S^1$ acts trivially on $U,U'$. Arguing as in Section \ref{sec:spin}, we have
\[
e^K_{Pin^G(2)}(H^+(X)_{\mathbb{C}}) e^K_{Pin^G(2)}(V'_{\mathbb{C}}) = \gamma e^K_{Pin^G(2)}(V_{\mathbb{C}})
\]
for some $\gamma \in K^0_{Pin^G(2)}(pt)$. Adapting the proof of Theorem \ref{thm:10on8}, we obtain:

\begin{theorem}\label{thm:10on8G}
Let $G$ act smoothly on $X$ preserving the isomorphism class of a spin structure $\mathfrak{s}$. Suppose that the $G$-action is spin-liftable. Then:
\[
\wedge^* H^+(X)_{\mathbb{C}} \otimes \wedge^* \psi^2(V') =  \eta  (\wedge^* \psi^2(V)) \in K^0_G(pt)
\]
for some $\eta \in K^0_G(pt)$.
\end{theorem}

As in Section \ref{sec:ktheory} we can improve this result by a factor of $2$ if the $K$-theoretic Euler class of $H^+(X)_\mathbb{C}$ vanishes (here we don't need to sacrifice torsion as $K^0_G(pt) = R(G)$ is torsion free).
\begin{theorem}\label{thm:10on8Gplus1}
Let $G$ act smoothly on $X$ preserving the isomorphism class of a spin structure $\mathfrak{s}$. Suppose that the $G$-action is spin-liftable and that $e^K_G(H^+(X)_\mathbb{C})=0$. Then in $K^0_G(pt)$ we have:
\[
\wedge^* H^+(X)_{\mathbb{C}} \otimes \wedge^* \psi^2(V') =  2 \eta  (\wedge^* \psi^2(V))
\]
for some $\eta \in K^0_G(pt)$.
\end{theorem}

\begin{remark}
Let $G$ act smoothly on $X$ preserving the isomorphism class of a spin$^c$-structure $\mathfrak{s}$ and suppose the $G$-action is spin$^c$-liftable. Then we can consider the restriction of the monopole map $f : S_{V,U} \to S_{V',U'}$ to the $G$-fixed point set gives an $S^1$-equivariant map
\[
f^G : S_{V^G,U^G} \to S_{{V'}^G , {U'}^G},
\]
where $V^G - {V'}^G = D^G \in K^0(pt)$, ${U'}^G - U^G = (H^+(X))^G \in KO^0(pt)$. Then the results of Sections \ref{sec:cohomology}-\ref{sec:spin} can be repeated with $f^G$ in place of $f$ and $H^+(X)^G, D^G$ in place of $H^+(X), D$. In this way, we arrive at the orbifold versions of Donaldson's and Furuta's theorems \cite[Theorem 3, Theorem 4]{ff}.
\end{remark}

\section{Case of $b_1(X)>0$}\label{sec:b1>0}

In this section we consider the case of a $4$-manifold $X$ with $b_1(X)>0$. Let $(E , \mathfrak{s}_{E/B})$ be as before a spin$^c$-family over $B$ with fibre $(X , \mathfrak{s})$. In \cite[Example 2.4]{bako} we showed that a finite dimensional approximation of the families monopole map for $(E , \mathfrak{s}_{E/B})$ can be constructed under the condition that the family $E \to B$ admits a section $x : B \to E$. Here we consider a variant of this construction which has the advantage of not requiring a section (it corresponds to pulling back the construction of \cite[Example 2.4]{bako} by the zero section $\zeta : B \to J$ of the Jacobian bundle).\\

Choose a smoothly varying fibrewise metric $g = \{g_b\}_{b\in B}$ on $E$ and a smoothly varying family of $U(1)$-connections $A = \{ A_b \}_{b \in B}$ for the determinant line of the spin$^c$-structure (one way to do this is to choose a globally defined connection on the total space of $E$ and define $A_b$ as the restriction of this connection to the fibre over $b$). Fix an integer $k >2$ and define the following Hilbert bundles over $B$:
\begin{equation*}
\begin{aligned}
&\mathbb{V} = L^2_k( S^+), &&\quad \mathbb{W} = L^2_k(\wedge^1 T^*X)_0, \\
&\mathbb{V}' = L^2_{k-1}(S^-), &&\quad \mathbb{W}' = L^2_{k-1}(\wedge^+ T^*X) \oplus L^2_{k-1}(\mathbb{R})_0.
\end{aligned}
\end{equation*}
In the above definitions $S^\pm$ denotes the spinor bundles, $L^2_k$ denotes the Sobolev space of $L^2_k$ sections, $L^2_{k-1}(\mathbb{R})_0$ denotes the subspace of sections $f \in L^2_{k-1}(\mathbb{R})$ satisfying $\int_X f dvol_X = 0$ and $L^2_k(\wedge^1 T^*X)_0$ is the subbundle of $L^2_k(\wedge^1 T^*X)$ consisting of $1$-forms $L^2$-orthogonal to the finite dimensional subbundle of harmonic $1$-forms.

We define the families Seiberg-Witten monopole map $\mathcal{F} : \mathbb{V} \oplus \mathbb{W} \to \mathbb{V}' \oplus \mathbb{W}'$ as
\[
\mathcal{F}( \psi , a ) = ( D_{A+ia} \psi , -iF^+_{A+ia} +i \sigma(\psi) + iF^+_A , d^*a)
\]
where $D_{A+ia}$ denotes the spin$^c$ Dirac operator associated to $A+ia$ and $\sigma(\psi)$ is the quadratic spinor term in the Seiberg-Witten equations. By essentially the same argument as in \cite[Example 2.4]{bako} one sees that $\mathcal{F}$ satisfies conditions conditions (M1)-(M7) of \cite[Section 2]{bako} and hence has a finite dimensional approximation
\[
f : S_{V,U} \to S_{V',U'}
\]
exactly as in Section \ref{sec:setup}. From this point onwards we can study the map $f$ in exactly the same way as we did for the case $b_1(X)=0$ in the previous sections. In summary we have: \\

{\em The results of the Sections \ref{sec:cohomology}-\ref{sec:spin} directly carry over to the case $b_1(X)>0$, without additional assumptions}.\footnote{In a previous version of this paper we required that the family $E \to B$ admits a section. I thank the referee for prompting me to examine this assumption further.} \\

Similarly, we can consider the $G$-equivariant monopole map for $4$-manifolds with $b_1(X)>0$. In this case a construction parallel to the families case works. The main difference compared to the families case is that now we need to choose a $G$-invariant reference connection $A$ in order to write down the monopole map. This can be done by first choosing any connection and then averaging over $G$, which is possible as connections form an affine space. Hence:\\

{\em The results of Section \ref{sec:equivariant} carry over to the case $b_1(X)>0$ without additional assumptions}.

\section{$\mathbb{Z}_2$-actions}\label{sec:z2}

Let $X$ be a compact, oriented, smooth $4$-manifold and let $f : X \to X$ be an orientation preserving involutive diffeomorphism. By averaging, there exists an $f$-invariant metric and using this metric we get an action of $\mathbb{Z}_2$ on $H^+(X)$.

Suppose that $f$ preserves the isomorphism class of a spin$^c$-structure $\mathfrak{s}$ on $X$. Then we can choose a lift $\tilde{f}$ of $f$ to the associated spinor bundles. Moreover, because the sequence $0 \to S^1 \to \hat{G} \to \mathbb{Z}_2 \to 0$ is always split, the lift can be chosen so that $\tilde{f}^2 = 1$ and this uniquely determines the lift up to an overall sign change $\tilde{f} \mapsto -\tilde{f}$. Let $d_{\pm}$ denote the virtual dimensions of the $\pm 1$ virtual eigenspaces of $\tilde{f}$ on $D$. Thus $d = d_+ + d_-$. Note that changing $\tilde{f}$ to $-\tilde{f}$ exchanges $d_+$ and $d_-$.

\begin{theorem}
Suppose that $f|_{H^+(X)} = -Id$. Then for any $f$-invariant spin$^c$-structure $\mathfrak{s}$, we have $c_1(\mathfrak{s})^2 \le \sigma(X)$ and $d_+,d_- \le 0$.
\end{theorem}
\begin{proof}
We let $G = \mathbb{Z}_2$ and consider the $G$-equivariant family over $B = \{pt\}$. The existence of an involutive lift $\tilde{f}$ shows that the $G$-action is liftable and $\widehat{G} = S^1 \times G$. Then $K^0_G(pt) = R[G]$ is generated by the two irreducible representations $\mathbb{C}_0, \mathbb{C}_1$, where $\mathbb{C}_i$ is the $1$-dimensional representation such that $f$ acts as $(-1)^i$. If $f|_{H^+(X)} = -id$, then $H^+(X)_{\mathbb{C}} = \mathbb{C}_1^{b_+(X)}$. Also $V = \mathbb{C}_0^{a_+} \oplus \mathbb{C}_1^{a_-}$, $V' = \mathbb{C}_0^{a'_+} \oplus \mathbb{C}_1^{a'_-}$ for some $a_+,a_-,a'_+,a'_-$, where $d_+ = a_+ - a'_+$, $d_- = a_- - a'_-$. Then Equation (\ref{equ:Geulerk2}) becomes:
\[
(1 - \mathbb{C}_1)^{b_+(X)} (1-\mathbb{C}_0 \xi^{-1})^{2a'_+}(1-\mathbb{C}_1 \xi^{-1})^{2a'_-} = \gamma (1-\mathbb{C}_0 \xi^{-1})^{2a_+}(1-\mathbb{C}_1 \xi^{-1})^{2a_-}
\]
for some $\gamma \in R[G][\xi , \xi^{-1}]$. The factors of $2$ in the exponents occur because $V_{\mathbb{C}} \cong V \oplus \overline{V} \cong \mathbb{C}_0^{2a_+} \oplus \mathbb{C}^{2a_-}$ and similarly for $V'$. Applying the homomorphism $tr_f : R[G] \to \mathbb{Z}$ which sends a representation $W$ to the trace $tr_f(W)$, we get the following equality
\[
2^{b_+(X)} (1-\xi^{-1})^{2a'_+}(1+\xi^{-1})^{2a'_-} = h (1-\xi^{-1})^{2a_+} (1+\xi^{-1})^{2a_-}
\]
where $h = tr_f(\gamma) \in \mathbb{Z}[\xi,\xi^{-1}]$ is a Laurent polynomial in $\xi$. Since $h$ is a Laurent polynomial it can be written in the form $h = \xi^{m}q(\xi^{-1})$, where $q(\xi^{-1})$ is a polynomial in $\xi^{-1}$ and $m$ is an integer. After re-arranging, we have
\[
q(\xi^{-1}) = 2^{b_+(X)} \xi^{-m}(1-\xi^{-1})^{-2d_+}(1+\xi^{-1})^{-2d_-}.
\]
But the right hand side is a polynomial in $\xi^{-1}$ only if $d_+,d_- \le 0$. It follows that $d = d_+ + d_- \le 0$ and since $d = (c_1(\mathfrak{s})^2 - \sigma(X))/8$, we get $c_1(\mathfrak{s})^2 \le \sigma(X)$.
\end{proof}

Now suppose that $X$ is spin and that $f$ preserves a spin structure $\mathfrak{s}$. Let $\tilde{f}$ denote a lift of $f$ to the associated principal $Spin(4)$-bundle. Then $\tilde{f}^2 = \pm 1$. Recall that $f$ is said to be of even type if $\tilde{f}^2=1$ and of odd type if $\tilde{f}^2=-1$. If $f$ is of even type then the fixed point set consists entirely of isolated points. If $f$ is of odd type then the fixed point set consists entirely of orientable surfaces.

\begin{theorem}\label{thm:spininvol}
Suppose that $f$ preserves a spin structure $\mathfrak{s}$ and $f$ is of even type. If $\sigma(X) < 0$, then $dim( H^+(X)^{\mathbb{Z}_2}) \ge 3$.
\end{theorem}
\begin{proof}
Let $\mathbb{Z}_2$ act on the sphere $S^{b_+(X)}$ by the antipodal map, so the quotient space is $\mathbb{RP}^{b_+(X)}$. We will take our family to be $E = X \times_{\mathbb{Z}_2} S^{b_+(X)} \to \mathbb{RP}^{b_+(X)}$. If $X$ has negative signature then $d > 0$, so by Corollary \ref{cor:eulerpin} we must have $w_i(H^+(X)) = 0$ for $i \ge b_+(X)-2$. However $H^+(X)$ is the flat bundle $S^{b_+(X)} \times_{\mathbb{Z}_2} H^+(X)$ associated to the action of $f$ on $H^+(X)$. As a representation of $\mathbb{Z}_2$, we have $H^+ = \mathbb{R}^{u} \oplus \mathbb{R}^{v}_-$, where $\mathbb{R}$ is the trivial representation, $\mathbb{R}_-$ is the sign representation and $u+v = b_+(X)$. Let $x \in H^1( \mathbb{RP}^{b_+(X)} ; \mathbb{Z}_2)$ be the generator of the cohomology of $\mathbb{RP}^{b_+(X)}$ with $\mathbb{Z}_2$-coefficients. Then the total Stiefel-Whitney class of $H^+(X)$ is easily seen to be $(1+x)^v$. In particular, $w_v(H^+(X)) \neq 0$. Hence $v \le b_+(X) -3$ and $u \ge 3$.
\end{proof}

\section{$\mathbb{Z}_p$-actions}\label{sec:zp}

Let $X$ be a compact, oriented, smooth $4$-manifold. Let $p$ be an odd prime and consider an action of $\mathbb{Z}_p$ on $X$ generated by a diffeomorphism $f : X \to X$ of order $p$. Clearly $f$ is orientation preserving since $p$ is odd. By averaging, there exists a $\mathbb{Z}_p$-invariant metric and using this metric we get an action of $\mathbb{Z}_p$ on $H^+(X)$.

Suppose that $f$ preserves the isomorphism class of a spin$^c$-structure $\mathfrak{s}$ on $X$. Then we can choose a lift $\tilde{f}$ of $f$ to the associated spinor bundles satisfying $\tilde{f}^p = 1$. Such a lift is uniquely determined up to multiplication by a $p$-th root of unity. For $0 \le j \le p-1$ we let $d_j$ denote the dimension of the $\omega^j$ virtual eigenspace of $\tilde{f}$ on $D$ where $\omega = exp(2\pi i /p)$. Thus $d = d_0 + d_1 + \cdots + d_{p-1}$. Note that changing the lift $\tilde{f}$ by a $p$-th root of unity has the effect of cyclically permuting $(d_0 , d_1, \dots , d_{p-1})$.

\begin{theorem}\label{thm:zp}
Suppose that $H^+(X)^{\mathbb{Z}_p} = 0$. Then for any $f$-invariant spin$^c$-structure $\mathfrak{s}$, we have $c_1(\mathfrak{s})^2 \le \sigma(X)$ and $d_j \le 0$ for each $j$.
\end{theorem}
\begin{proof}
We let $G = \mathbb{Z}_p$ and consider the $G$-equivariant family over $B = \{pt\}$. By the existence of a lift $\tilde{f}$ satisfying $\tilde{f}^p = 1$, the $G$-action is liftable and $\widehat{G} = S^1 \times G$. Then $K^0_G(pt) = R[G]$ is generated by the irreducible representations $\mathbb{C}_j, 0 \le j \le p-1$, where $\mathbb{C}_j$ is the $1$-dimensional representation such that $f$ acts as $\omega^j$. As a representation of $G$, we have:
\[
H^+(X)_{\mathbb{C}} = \mathbb{C}_0^{h_0} \oplus \mathbb{C}_1^{h_1} \oplus \cdots \oplus \mathbb{C}_{p-1}^{h_{p-1}}
\]
for some non-negative integers $h_0, \dots , h_{p-1}$. Moreover $h_j = h_{p-j}$ since $H^+(X)_{\mathbb{C}}$ is the complexification of a real representation. Note also that $h_0 = 0$ because of the assumption that $H^+(X)^{\mathbb{Z}_p} = 0$. Similarly, we write $V,V'$ as:
\[
V = \mathbb{C}_0^{a_0} \oplus \mathbb{C}_1^{a_1} \oplus \cdots \oplus \mathbb{C}_{p-1}^{a_{p-1}}, \quad V' = \mathbb{C}_0^{a'_0} \oplus \mathbb{C}_1^{a'_1} \oplus \cdots \oplus \mathbb{C}_{p-1}^{a'_{p-1}}.
\]
Then $d_j = a_j - a'_j$. Equation (\ref{equ:Geulerk2}) becomes:
\[
\prod_{j=1}^{p-1}(1-\mathbb{C}_j)^{h_j} \prod_{j=0}^{p-1} (1 - \mathbb{C}_j \xi^{-1})^{a'_j+a'_{p-j}} = \gamma \prod_{j=0}^{p-1} (1 - \mathbb{C}_j \xi^{-1})^{a_j+a_{p-j}}
\]
for some $\gamma \in R[G][\xi , \xi^{-1}]$. Applying the homomorphism $tr_f : R[G] \to \mathbb{Z}[\omega]$ which sends a representation $W$ to the trace $tr_f(W)$, we get the following equality:
\[
\prod_{j=1}^{p-1}(1-\omega_j)^{h_j} \prod_{j=0}^{p-1} (1 - \omega_j \xi^{-1})^{a'_j+a'_{p-j}} = \gamma \prod_{j=0}^{p-1} (1 - \omega_j \xi^{-1})^{a_j+a_{p-j}}
\]
where $h = tr_f(\gamma) \in \mathbb{Z}[\omega][\xi,\xi^{-1}]$ is a Laurent polynomial in $\xi$. Since $h$ is a Laurent polynomial it can be written in the form $h = \xi^{m}q(\xi^{-1})$, where $q(\xi^{-1})$ is a polynomial in $\xi^{-1}$ and $m$ is an integer. After re-arranging, we have
\[
q(\xi^{-1}) = \prod_{j=1}^{p-1}(1-\omega_j)^{h_j} \prod_{j=0}^{p-1} (1 - \omega_j \xi^{-1})^{-d_j-d_{p-j}}
\]
But the right hand side is a polynomial in $\xi^{-1}$ only if $d_j+d_{p-j} \le 0$ for each $j$ (to see this, note that the right hand side is a rational function in $\xi^{-1}$. It can only be a polynomial if its reduced form has no factors in the denominator). In particular $d_0 \le 0$. Since changing the lift of $f$ by a $p$-th root of unity cyclically permutes the $d_j$, we must have $d_j \le 0$ for each $j$. It follows also that $d = d_0 + \cdots + d_{p-1} \le 0$ and hence $c_1(\mathfrak{s})^2 \le \sigma(X)$.
\end{proof}

Now suppose that $X$ is spin and that $f$ preserves a spin structure $\mathfrak{s}$. Let $\tilde{f}$ denote a lift of $f$ to the associated principal $Spin(4)$-bundle. Then $\tilde{f}^p = \pm 1$. Replacing $\tilde{f}$ by $-\tilde{f}$ if necessary, we can assume that $\tilde{f}^p = 1$ and this uniquely determines $\tilde{f}$.

\begin{theorem}
Suppose that $f$ preserves a spin structure $\mathfrak{s}$. If $H^+(X)^{\mathbb{Z}_p} \neq 0$ then
\[
dim_{\mathbb{R}}( H^+(X)^{\mathbb{Z}_p}) \ge d_0+1 = dim_{\mathbb{C}}( D^{\mathbb{Z}_p})+1.
\]
\end{theorem}
\begin{proof}
This is a special case of \cite[Theorem 4]{ff} applied to the orbifold quotient $X/\mathbb{Z}_p$. Alternatively it can be deduced from an application of Theorem \ref{thm:10on8Gplus1}.
\end{proof}

In \cite{bar} we gave examples of $\mathbb{Z}_2$-actions on the intersection lattice of $4$-manifolds which could be realised by a continuous action but not smoothly. Using the results of this section we can show similar results for $\mathbb{Z}_p$-actions where $p$ is any odd prime.

\begin{theorem}
Let $p$ be an odd prime and let $b,g$ be integers with $g(p-1) \ge 3bp$ and $b \ge 1$. Let $X$ be the topological $4$-manifold $X = \# g(p-1) (S^2 \times S^2) \# 2bp(-E_8)$. Then $H^2(X ; \mathbb{Z})$ admits an isometry $f : H^2(X ; \mathbb{Z}) \to H^2(X ; \mathbb{Z})$ of order $p$ with the following properties:
\begin{itemize}
\item[(i)]{$f$ can be realised by the induced action of a continuous, locally linear $\mathbb{Z}_p$-action on $X$.}
\item[(ii)]{If $g(p-1) > 3bp$ then $f$ can be realised by the induced action of a diffeomorphism $X \to X$, where the smooth structure is obtained by viewing $X$ as $\# (g(p-1)-3bp)(S^2 \times S^2) \# pb(K3)$.}
\item[(iii)]{$f$ can not be induced by a smooth $\mathbb{Z}_p$-action for any smooth structure on $X$.}
\end{itemize}
\end{theorem}
\begin{proof}
We construct $X$ together with a continuous $\mathbb{Z}_p$-action. Let $X_0$ denote the $p$-fold cyclic cover of $S^4$ branched along a genus $g$ oriented surface $\Sigma \to S^4$ given its standard embedding (so that $\Sigma$ bounds a genus $g$ handlebody in $S^4$). Then by \cite[Corollary 4.3]{ak} $X_0$ is diffeomorphic to $\#g(p-1)(S^2 \times S^2)$. Since $X_0$ is a cyclic branched cover it admits a smooth $\mathbb{Z}_p$-action. Let $\varphi_0 : X_0 \to X_0$ be a generator of this action and let $\pi : X_0 \to S^4$ denote the covering map. The induced map $\varphi_0 : H^2(X_0 ; \mathbb{Z}) \to H^2(X_0 ; \mathbb{Z})$ satisfies $1 + \varphi_0 + \varphi^2_0 + \cdots + \varphi_0^{p-1} = \pi^* \pi_*$, where $\pi_* : H^2(X_0 ; \mathbb{Z}) \to H^2(S^4 ; \mathbb{Z})$ is the pushforward map in cohomology. Then since $H^2(S^4 ; \mathbb{Z}) = 0$ it follows that $1 + \varphi_0 + \cdots + \varphi_0^{p-1} = 0$. It also follows that $H^2( X_0 ; \mathbb{R})^{\mathbb{Z}_p} = 0$ for if $v \in H^2(X_0 ; \mathbb{R})$ is fixed by $\varphi_0$, then $pv = (1+\varphi_0 + \cdots + \varphi_0^{p-1})v = 0$, so $v = 0$. In particular we deduce that $H^+(X_0)^{\mathbb{Z}_p} = 0$.

Next, choose a point $x \in X_0$ which is not fixed by $\varphi_0$. Let $X$ be obtained from $X_0$ by attaching $\# 2b(-E_8)$ to each point in the $\mathbb{Z}_p$-orbit of $x$. Then $X = \# g(p-1) (S^2 \times S^2) \# 2bp(-E_8)$. The $\mathbb{Z}_p$-action on $X_0$ generated by $\varphi_0$ determines a corresponding $\mathbb{Z}_p$-action on $X$. Let $\varphi : X \to X$ denote the corresponding generator of this action and let $f = \varphi^* : H^2(X ; \mathbb{Z}) \to H^2(X ; \mathbb{Z})$ be the induced isometry of $H^2(X ; \mathbb{Z})$. Then $f$ is realised by the continuous $\mathbb{Z}_p$-action generated by $\varphi$. By construction $\varphi$ is locally linear, proving (i). 

The connected sum decomposition of $X$ gives an identification $H^2(X ; \mathbb{Z}) = g(p-1) H \oplus 2bp (-E_8)$, where $H$ denotes the intersection form of $S^2 \times S^2$. Since $-E_8$ is negative definite, we find that $H^+(X)^{\mathbb{Z}_p} \cong H^+(X_0)^{\mathbb{Z}_p} = 0$. Taking $\mathfrak{s}$ to be the unique spin structure on $X$, Theorem \ref{thm:zp} implies that there does not exist a smooth $\mathbb{Z}_p$-action realising $f$ for any smooth structure on $X$, proving (iii). Lastly if $g(p-1) \ge 3bp$ then $X$ admits at least one smooth structure since we can write $X$ as $\# (g(p-1)-3bp)(S^2 \times S^2) \# pb(K3)$. Moreover if $g(p-1) > 3bp$ then by \cite[Theorem 2]{wall}, every isometry of $H^2(X ; \mathbb{Z})$ is realised by a diffeomorphism. In particular $f$ is realised by some diffeomorphism of $X$, which proves (ii). 
\end{proof}

\section{Application to non-smoothable families}\label{sec:nonsmooth}

In this Section, we consider an application of our main obstruction results to the existence of non-smoothable families of $4$-manifolds. For a smooth $4$-manifold $X$ we let $Homeo(X)$ denote the group of orientation preserving homeomorphisms of $X$ with the $\mathcal{C}^0$-topology and $Diff(X)$ the group of orientation preserving diffeomorphisms of $X$ with the $\mathcal{C}^\infty$-topology. The natural inclusion $Diff(X) \to Homeo(X)$ is continuous, but not a closed embedding.

\begin{definition}\label{def:smoothable}
Let $B$ be a compact smooth manifold. 
\begin{itemize}
\item{By a {\em continuous family of $4$-manifolds over $B$ with fibres homeomorphic to $X$}, we mean a topological fibre bundle $\pi : E \to B$ with fibres homeomorphic to $X$ with transition functions in $Homeo(X)$.}
\item{We say that $\pi : E \to B$ is {\em smoothable with fibres diffeomorphic to $X$} if there exists a reduction of structure group of $E$ to $Diff(X)$.}
\end{itemize}
\end{definition}

To be more explicit, a continuous family $\pi : E \to B$ with fibres homeomorphic to $X$ is smoothable with fibres diffeomorphic to $X$ if $E$ can be constructed from an open cover $\{ U_i \}$ of $B$ and with transition functions given by {\em continuous} maps $g_{ij} : U_{ij} \to Diff(X)$. The underlying topological fibre bundle is then given by composing the transition functions $g_{ij}$ with the inclusion $Diff(X) \to Homeo(X)$.\\

As explained in \cite[\textsection 4.2]{bako}, it follows from a result of M\"uller-Wockel \cite{muwo} that $E$ is smoothable with fibres diffeomorphic to $X$ if and only if $E$ admits the structure of a smooth manifold such that $\pi : E \to B$ is a submersion and the fibres of $E$ with their induced smooth structure are diffeomorphic to $X$ (this is the notion of a smooth family that we have been using throughout the paper).

\begin{remark}
All the $4$-manifolds we consider in this section will be oriented and have non-zero signature. In this case every homeomorphism is automatically orientation preserving.
\end{remark}

\begin{theorem}\label{thm:nonsmooth1}
Let $X$ be a compact, smooth, simply-connected $4$-manifold with $| \sigma(X) | > 8$ and indefinite intersection form. Then:
\begin{itemize}
\item{If $X$ is non-spin, there exists a topological fibre bundle $E \to B$ with fibres homeomorphic to $X$ and $B$ is a torus of dimension $\min \{b_+(X),b_-(X)\}$ such that $E$ is non-smoothable.}
\item{If $X$ is spin, there exists a topological fibre bundle $E \to B$ with fibres homeomorphic to $X$ and $B$ is a torus of dimension $\min \{b_+(X),b_-(X))\}-2$ such that $E$ is non-smoothable.}
\end{itemize}
\end{theorem}
\begin{remark}
This theorem is a generalisation of \cite[Corollary 1.5]{kkn}.
\end{remark}
\begin{proof}
We consider the spin case first. We may as well assume $\sigma(X) < 0$ and then $X$ is homeomorphic to $\# a(S^2 \times S^2) \# 2b(-E_8)$ for some $a \ge 0$, $b > 0$. Note that $a = b_+(X) \ge 3$ by the $10/8$ inequality. Let $(S^2 \times S^2)_1 , \dots , (S^2 \times S^2)_a$ denote the $a$ summands of $\# a(S^2 \times S^2)$. Let $f_j : (S^2 \times S^2)_j \to (S^2 \times S^2)_j$ be an orientation preserving diffeomorphism which acts as $-1$ on $H^+( (S^2 \times S^2)_j )$. Applying an isotopy to $f_j$ if necessary, we can assume that there exists an open neighbourhood $N_j$ of a point of $(S^2 \times S^2)_j$ on which $f_j$ acts as the identity. For $1 \le j < a$,  we attach $(S^2 \times S^2)_j$ to $(S^2 \times S^2)_{j+1}$ by removing open balls from $N_j, N_{j+1}$ and identifying their boundaries. In this way $f_1, \dots , f_a$ act as commuting diffeomorphisms on $\# a (S^2 \times S^2)$. In a similar manner we attach $2b$ copies of $(-E_8)$ to $\# a(S^2 \times S^2)$ via handles that end in neighbourhoods where the $f_j$ act trivially. In this way we have constructed a continuous orientation preserving $\mathbb{Z}^a$-action on $X$. Forgetting the action of $f_{a-1}, f_a$, we get a continuous $\mathbb{Z}^{a-2}$-action. Let $B = T^{a-2}$ be the torus of dimension $a-2 = b_+(X)-2$ and let $\pi : E \to B$ be the mapping torus of the $\mathbb{Z}^{a-2}$-action. Arguing as in \cite[Lemma 2.8]{kkn}, one sees that $E$ admits a topological spin structure (see \cite{nak} for an explanation of topological spin structures).\\

Now suppose $\pi : E \to B$ is smoothable. In this case, the topological spin structure determines a spin structure on $E$ in the usual sense and it follows that the vertical tangent bundle of $E$ admits a spin structure which restricts to the unique spin structure on $X$. It is straightforward to see that $w_{b_+(X)-2}(H^+(X)) \neq 0$, which contradicts Corollary \ref{cor:eulerpin}. Therefore $E$ is not smoothable.\\

Now we consider the non-spin case. We may as well assume $\sigma(X) < -8$ and then $X$ is homeomorphic to $\# a(\mathbb{CP}^2) \# (a+b+9)(\overline{\mathbb{CP}^2})$ for some $a,b \ge 0$. Note that $a = b_+(X) > 0$, since $X$ is indefinite. We let $\overline{\mathbb{CP}^2}_{fake}$ denote the compact simply-connected topological $4$-manifold whose intersection form is $(-1)$ and whose Kirby-Siebenmann class is non-zero. Then $X$ is homeomorphic to $\# a(S^2 \times S^2) \# b \overline{\mathbb{CP}^2} \# (-E_8) \# \overline{\mathbb{CP}^2}_{fake}$. As in the spin case, we construct commuting homeomorphisms $f_1, \dots , f_a$ acting on the $a$ copies of $(S^2 \times S^2)$ and trivially on the remaining factors. In this way we obtain a continuous orientation preserving action of $\mathbb{Z}^{a}$ on $X$. Let $B = T^a$ be the torus of dimension $a = b_+(X)$ and let $\pi : E \to B$ be the mapping torus of the $\mathbb{Z}^a$-action. Using an argument similar to the spin case, we have that $E$ admits a topological spin$^c$-structure. Moreover this topological spin$^c$-structure can be chosen so that its restriction to any fibre has characteristic $c \in H^2( X ; \mathbb{Z})$ which is zero on the $a(S^2 \times S^2)$ and $(-E_8)$ factors and satisfies $c^2 = -(b+1)$ (each copy of $\overline{\mathbb{CP}^2}$ and the copy of $\overline{\mathbb{CP}^2}_{fake}$ each contributing $-1$).\\

Suppose $E$ is smoothable. Then similar to the spin case, the topological spin$^c$-structure gives a spin$^c$-structure $\mathfrak{s}$ on the vertical tangent bundle in the usual sense. We have $c_1(\mathfrak{s})^2 = c^2 = -(b+1) > \sigma(X) = -(b+9)$. But also it is straightforward to see that $w_{b_+(X)}(H^+(X)) \neq 0$, which contradicts Theorem \ref{thm:euler}. Hence $E$ is not smoothable.
\end{proof}

\begin{corollary}
Let $X$ be a compact, smooth, simply-connected $4$-manifold with $| \sigma(X) | > 8$ and indefinite intersection form. Then the inclusion $Diff(X) \to Homeo(X)$ is not a weak homotopy equivalence. More precisely:
\begin{itemize}
\item{If $X$ is non-spin then $\pi_j( Diff(X)) \to \pi_j( Homeo(X))$ is not an isomorphism for some $j \le \min\{b_+(X),b_-(X)\}-1$.}
\item{If $X$ is spin then $\pi_j( Diff(X)) \to \pi_j( Homeo(X))$ is not an isomorphism for some $j \le \min\{b_+(X),b_-(X)\}-3$.}
\end{itemize}
\end{corollary}
\begin{proof}
The homomorphism $Diff(X) \to Homeo(X)$ induces a continous map \linebreak $BDiff(X) \to BHomeo(X)$. A continuous family $\pi : E \to B$ with fibres homeomorphic to $X$ is equivalent to giving a homotopy class of map $f : B \to BHomeo(X)$. Then $E$ is smoothable if and only if there exists a lift of $f$ to $BDiff(X)$:
\[
\xymatrix{
& BDiff(X) \ar[d] \\
B \ar[r]_-f \ar@{-->}[ur]& BHomeo(X)
}
\]
Let $b = \min\{ b_+(X) , b_-(X)\}$ if $X$ is non-spin and $b = \min\{ b_+(X) , b_-(X)\} - 2$ is $X$ is spin. Let $B = T^b$ be a torus of dimension $b$. Then by Theorem \ref{thm:nonsmooth1}, there exists a map $f : T^b \to BHomeo(X)$ which does not lift to $BDiff(X)$. By standard obstruction theory, it follows that $\pi_{j}( BDiff(X)) \to \pi_{j}(BHomeo(X))$ is not an isomorphism for some $j$ in the range $1 \le j \le b$. But since $\pi_j( Diff(X)) \cong \pi_{j+1}(BDiff(X))$ and $\pi_j( Homeo(X)) \cong \pi_{j+1}(BHomeo(X))$, we get that \linebreak $\pi_j( Diff(X)) \to \pi_j( Homeo(X))$ is not an isomorphism for some $j \le b-1$. In particular, $Diff(X) \to Homeo(X)$ is not a weak homotopy equivalence.
\end{proof}

\begin{remark}
It is interesting to compare this result with manifolds of lower dimension. As observed in \cite{kkn}, if $M$ is a compact oriented smooth manifold of dimension $\le 3$, then $Diff(M) \to Homeo(M)$ is known to be a weak homotopy equivalence.
\end{remark}


\bibliographystyle{amsplain}

\end{document}